\documentclass[12pt,a4paper]{amsart}
\usepackage{amsfonts}
\usepackage{amssymb}
\usepackage{amsmath}
\usepackage{enumerate}
\usepackage{indentfirst}
\usepackage{setspace}
\usepackage{fullpage}
\usepackage{graphicx}

\RequirePackage{amsthm}
\newcommand{\cl}{\overline}
\newcommand{\set}[2]{\mbox{$\left\{#1:\;#2\right\}$}}
\def\R{\mathbb{R}}
\newtheorem{thm}{Theorem}[section]
\newtheorem{cor}[thm]{Corollary}

\newtheorem{lem}[thm]{Lemma}

\newtheorem{prop}[thm]{Proposition}
\theoremstyle{definition}
\newtheorem{defin}[thm]{Defnition}
\newtheorem{rem}[thm]{Remark}
\newtheorem{exa}[thm]{Example}

\def\<#1>{\left\langle #1\right\rangle}
\numberwithin{equation}{section}

\newcommand{\E}{{\mathcal E}}

\newcommand{\tm}{{t_{\text{max}}}}
\newcommand{\fl}[1]{\left\lfloor #1\right\rfloor}

\DeclareMathOperator*{\einf}{ess\,inf}

\begin{document}
\title[Positive solutions of parabolic systems]{Multiple positive solutions of parabolic systems with nonlinear, nonlocal initial conditions}

\author[G. Infante]{Gennaro Infante}
\address{Gennaro Infante, Dipartimento di Matematica ed Informatica,
Universit\`{a} della Calabria, 87036 Arcavacata di Rende, Cosenza, Italy}
\email{gennaro.infante@unical.it}
\author[M. Maciejewski]{Mateusz Maciejewski}
\address{Mateusz Maciejewski, Nicolaus Copernicus University, Faculty of
Mathematics and Computer Science, ul. Chopina 12/18, 87-100 Toru\'n, Poland}
\email{Mateusz.Maciejewski@mat.umk.pl}
\subjclass[2010]{Primary 35K51, secondary 35B09, 35B45, 35D05, 47H10}
\keywords{Fixed point index, parabolic system, nonlocal initial condition, positive mild solution, cone, weak Harnack inequality, multiplicity, nonexistence}

\begin{abstract}
In this paper we study the existence, localization and multiplicity of positive solutions for parabolic systems with nonlocal initial conditions.
In order to do this, we extend an abstract theory that was recently developed by the authors jointly with Radu Precup, related to the existence of fixed points of nonlinear operators satisfying some upper and lower bounds. Our main tool is the Granas fixed point index theory. We also provide a non-existence result and some examples to illustrate our theory.
\end{abstract}
\maketitle

\section{Introduction}\label{sect:intro}
In this paper we deal with the existence, non-existence and localization of positive solutions of the following system of parabolic equations subject to nonlinear, nonlocal initial conditions
\begin{equation}\label{eq:parabolic-intro}
\begin{cases}
u_t-\Delta u = f(t,x,u,v), & (t,x)\in (0,\tm)\times\Omega,\\
v_t-\Delta v = g(t,x,u,v), & (t,x)\in (0,\tm)\times\Omega,\\
u(t,x)=v(t,x)=0, & (t,x)\in(0,\tm)\times\partial\Omega,\\
u(0,\cdot)=\alpha(u,v),\\
v(0,\cdot)=\beta(u,v),
\end{cases}
\end{equation}
where $\Omega \subset \mathbb{R}^{m}$ is a bounded Dirichlet regular domain (i.e. for all $\varphi\in C(\partial \Omega)$ there exists $u\in C(\cl\Omega)$ such that $u|\partial\Omega=\varphi$ and $\Delta u=0$ in a distributional sense, see \cite[Definition 6.1.1]{b:ABHN}), $f,g\colon (0,\tm)\times\Omega\times\R_+\times\R_+\to\R_+$ are continuous functions and
\begin{equation}\label{init-con}
\alpha(u,v)=G_1\left(\int_0^{\tm} g_1(u(t),v(t))d\mu_1(t)\right),\quad \beta(u,v)=G_2\left(\int_0^{\tm} g_2(u(t),v(t))d\mu_2(t)\right),
\end{equation}
where $g_1,g_2\colon\R_+^2\to\R_+$, $G_1,G_2\colon\R_+\to\R_+$ are continuous with $g_i(0,0)=G_i(0)=0$ for $i\in\{1,2\}$ and $\mu_1,\mu_2$ are finite (positive) Borel measures on $[0,\tm]$ such that $$\mu_1(\{0\})=\mu_2(\{0\})=0.$$ Note that the initial conditions \eqref{init-con} cover a variety of cases, a \emph{particular} example being
\begin{equation}\label{eq:ntime-intro}
\alpha(u,v)(x)=\int_0^{\tm} u(t,x)dt,\quad \beta(u,v)(x)=\int_0^{\tm} v(t,x)dt.
\end{equation}
Note that the system \eqref{eq:parabolic-intro} can be applied to describe physical phenomena in which it is possible to measure the sums of amounts of substances according to formulae of the type~\eqref {eq:ntime-intro} or
\begin{equation}\label{eq:alpha-multipoint}
\alpha(u,v)=\sum_{s=1}^{k} \alpha_su(t_s),\quad \beta(u,v)=\sum_{s=1}^{r} \beta_sv(t'_s),
\end{equation}
where $0<t_1<\ldots<t_k$, $0<t'_1<\ldots<t'_r$ and $\alpha_s,\beta_s>0$. For example, the system~\eqref{eq:parabolic-intro} with the nonlocal initial conditions \eqref{eq:alpha-multipoint} can be used to model a reaction-diffusion process of a little amount of gases in a transparent tube, in a more appropriate way than with the usual initial-value condition; for more insight on the physical interpretation of this system see the paper by Byszewski~\cite{Bys1}. We mention that a physical motivation for the integral form of the initial condition is presented in \cite{Bys2,Bys3,b:olm-rob}. Moreover, the system \eqref{eq:parabolic-intro} can be used to describe the periodic solutions in the case $\alpha(u,v)=u(T)$, $\beta(u,v)=v(T)$, where $T>0$ and $f$ and $g$ are $T$-periodic.
Furthermore, a number of applications of nonlocal problems for evolution equations are illustrated in Section 10.2 of \cite{b:McKib}.

Initial nonlocal conditions have been investigated in a variety of settings, for example in the case of multi-point \cite{b:Chab},
integral \cite{b:tc-rp-pr,b:Deng,b:olm-rob,b:pao,b:ras-kar1,b:ras-kar2} and nonlinear conditions \cite{b:b-t-v,b:Bou1,b:Bou2,b:Bou3}, see also the recent review \cite{b:Stik}.

In a recent paper \cite{b:gi-mm-rp} the authors investigated the existence, localization and multiplicity of positive solutions of systems of $(p,q)$-Laplacian equations subject to Dirichlet boundary conditions. The main tool in  \cite{b:gi-mm-rp} is the development of a
general abstract framework for the existence of fixed points of nonlinear operators acting on cones that satisfy an inequality of Harnack-type. Within this setting, the authors of \cite{b:gi-mm-rp} used the Granas fixed point index (see for example \cite{b:Deim,b:gd}) and, in order to compute the index, they used some estimates from above, using the norm, and from below, utilizing a seminorm. Here we \emph{extend} the theoretical results of  \cite{b:gi-mm-rp} to a more general setting; this generalization is \emph{apparently}
 simple but fruitful and  is motivated by the application to the parabolic system \eqref{eq:parabolic-intro}.
In particular, we replace the use of the seminorm with the use of a more general positively homogeneous functional and, moreover, we relax the assumptions on the cone. The Remarks \ref{comp}, \ref{rem:choiceC0}, 
and \ref{rem:WhySuchHarnack} illustrate in details the differences between the two theoretical approaches and their applicability.
We point out that our new approach is quite general and covers, \emph{as a special case}, the system \eqref{eq:parabolic-intro}.

The problem of \emph{one} parabolic equation with nonlocal, \emph{linear} initial condition was stated and discussed in \cite{b:rp-par}. By modifying the theoretical setting as described above,
we overcome the difficulties arisen in \cite{b:rp-par}, obtaining the results predicted by~\cite{b:rp-par}.

In contrast with the paper \cite{b:gi-mm-rp}, where the space $L^\infty$ with an integral seminorm was used, here, in order to seek mild solutions of our problem, we use the classical space of continuous functions, with a very natural positively homogeneous functional, namely the minimum on a suitable subset. A similar idea has been used with success in the context of ordinary differential equations and integral equations, see for example~\cite{b:Deim,b:guolak,b:legg-will}. In our case, a key role for our multiplicity results is played by a weak Harnack-type inequality, see Remark~\ref{rem:Harnack}.

In the case of the system~\eqref{eq:parabolic-intro} we obtain existence, localization, multiplicity and non-existence of positive mild solutions.

We illustrate in two examples the applicability of our results and we show that the constants that occur in our theory can be computed.

\section{Abstract existence theorems}\label{sect:abstr}
In this Section we generalize the abstract results of Section 2 of \cite{b:gi-mm-rp}. Although these results are motivated by the solvability of the parabolic system~\eqref{eq:parabolic-intro}, we present them in a greater generality, as we believe that they are of independent interest, since they can be applied in other contexts, whenever an abstract Harnack-type inequality is available.

For $i=1,2$, let $(E_{i},|\cdot|_i)$ be a Banach space and let $\fl{\cdot}_i$ be a given positively homogeneous continuous functional on $E_i$. In what follows, we omit the subscript in $\fl\cdot_i$, when confusion is unlikely.

Let also $G_i\subset E_i$ be closed convex wedges, which is understood to mean that
\[\lambda u+\mu v\in G_i\text{ for all }u,v\in G_i\text{ and }\lambda,\mu\geq 0.\] Moreover, let $K_i\subset G_i$ be closed convex cones, which means that $K_i$ are closed convex wedges such that $K_i\cap(-K_i)=\{0\}$. The wedges induce the natural semiorders $\preceq$ on $E_i$ in the following way: $$u\preceq v\ \text{if and only if}\ v-u\in G_i,\, u,v\in E_i.$$ By a semiorder we mean that the relation $\preceq$ is reflexive and transitive, but not necessarily antisymmetric.

We assume that the functionals $\fl{\cdot}_i$ are monotone on $K_i$ with respect to the semiorder $\preceq$, that is for $i\in\{1,2\}$
we have
\begin{equation}\label{eq:fl monot}
\text{if }u\preceq v\text{ then }\fl u \leq\fl v\text{ for }u,v\in K_i.
\end{equation}
In particular we have $\fl u\geq 0$ for $u\in K_i$.

We assume that there exist some elements $\psi_i\in K_i$ such that $|\psi_i|=1$ and for $i\in\{1,2\}$
\begin{equation}\label{eq: prec-max} u\preceq |u|\psi_i\text{ for all }u\in K_i.\end{equation}

Note that \eqref{eq:fl monot} and  \eqref{eq: prec-max} yield
\begin{equation}\label{eq:rel-fl-norm}
\fl u_i\leq \fl{\psi} |u|\text{ for all } u\in K_i.\end{equation}
In particular, we have $\fl{\psi_i}>0$ if $\fl\cdot$ is nonzero.

In what follows by the \emph{compactness} of a continuous operator we mean the
relative compactness of its range. By the \emph{complete continuity} of a
continuous operator we mean the relative compactness of the image of every
bounded set of the domain.

We seek the fixed points of a completely continuous operator
\[N:=(N_1,N_2)\colon K_1\times K_2\to K_1\times K_2,\]
that is $(u,v)\in K_1\times K_2$ such that $N(u,v) =(u,v)$.

We shall discuss not only the existence, but also the localization and multiplicity
of the solutions of the nonlinear equation $N(u,v) =(u,v)$. In order to do this, 
we utilize the Granas fixed point index, $\mathrm{ind}_C(f,U)$, which
roughly speaking, is the algebraic count of the fixed points of $f$
in the set $U$. The formal definition of the index 
involves the \emph{Leray-Schauder degree} and retractions, for more information on the index and its applications
we refer the reader to~\cite{b:Deim,b:gd}.

The next Proposition describes some of the useful properties of the index,
for details see Theorem 6.2, Chapter 12 of \cite{b:gd}.

\begin{prop}
\label{fpi-pro} Let $C$ be a closed convex subset of a Banach space, $%
U\subset C$ be open in $C$ and $f\colon \overline{U}\rightarrow C$ be a
compact map with no fixed points on the boundary $\partial U$ of $U.$ Then
the fixed point index has the following properties:

\emph{(i)} (Existence) If $\mathrm{ind}_{C}(f,U)\neq 0$ then $\mathrm{fix}%
(f)\neq \emptyset $, where $\mathrm{fix}f=%
\mbox{$\left\{x\in\bar
U:\;f(x)=x\right\}$}$.

\emph{(ii)} (Additivity) If $\mathrm{fix}f\subset U_1 \cup U_2\subset U$
with $U_1, U_2$ open in $C$ and disjoint, then
\begin{equation*}
\mathrm{ind}_{C}(f,U)=\mathrm{ind}_{C}(f,U_1 )+\mathrm{ind}_{C}(f,U_2).
\end{equation*}

\emph{(iii)} (Homotopy invariance) If $h\colon \overline{U}\times \lbrack
0,1]\rightarrow C$ is a compact homotopy such that $h(u,t)\neq u$ for $u\in
\partial U$ and $t\in \lbrack 0,1]$ then
\begin{equation*}
\mathrm{ind}_{C}(h(\cdot,0),U)=\mathrm{ind}_{C}(h(\cdot,1),U).
\end{equation*}

\emph{(iv)} (Normalization) If $f$ is a constant map, with $f(u)=u_{0}$ for
every $u\in \overline{U},$ then
\begin{equation*}
\mathrm{ind}_{C}(f,U)=%
\begin{cases}
1,\quad \text{if}\ u_{0}\in U \\
0,\quad \text{if}\ u_{0}\notin \overline{U}.%
\end{cases}%
\end{equation*}

In particular, $\mathrm{ind}_{C}(f,C)=1$ for every compact function $%
f:C\rightarrow C,$ since $f$ is homotopic to any $u_{0}\in C,\ $by the
convexity of $C$ (take $h\left( u,t\right) =tf\left( u\right) +\left(
1-t\right) u_{0}$).
\end{prop}

\subsection{Fixed point results}
We begin with two theorems on the existence and localization of one
solution of the operator equation $N(u,v)=(u,v)$. Set $K=K_1\times K_2$ and
\[C=C(R_1,R_2):=\left\{ \left( u,v\right) \in K_1\times K_2:\left| u\right| \leq R_1,\ \left| v\right| \leq R_2\right\},\]
for some fixed numbers $R_1, R_2$.

The first Theorem is a generalization of Theorem 2.17 of \cite{b:gi-mm-rp}.

\begin{thm}\label{thm:existence 1a}
Assume that there exist numbers $r_i,R_i$, $i=1,2$ with $0<r_i<\fl{\psi_i}R_i$ such that
\begin{equation}\label{eq:ind0}
\inf\limits_{\substack{ (u,v)\in C  \\ \fl u  =r_1,\fl v  \leq
r_2}}\fl{ N_1 (u,v)} > r_1,\ \
\inf\limits_{\substack{ (u,v)\in C  \\ \fl u  \leq r_1,\fl v
=r_2}}\fl{ N_2(u,v)} > r_2,
\end{equation}
and
\begin{equation}
\sup\limits_{(u,v)\in C}|N_{i}(u,v)|\leq R_{i}\ \ \ \ (i=1,2).  \label{eq:ind1}
\end{equation}

Then $N$ has at least one fixed point $(u,v)\in K_1\times K_2$ such that $|u|\leq R_1$,
$|v|\leq R_2$ and either $\fl u  > r_1 $ or $\fl v  >r_2$.
\end{thm}
\begin{proof}
The assumption \eqref{eq:ind1} implies that $N(C)\subset C$. Therefore, by
Proposition~\ref{fpi-pro}, we obtain $\mathrm{ind}_{C}(N,C)=1$. Let
\begin{equation*}
U:=\left\{ \left( u,v\right) \in C:\fl u <r_1,\ \fl v <r_2\right\}.
\end{equation*}%
This is an open set, whose boundary $\partial U$ with respect to $C$ is equal to $\partial
U=A_1 \cup A_2,$ where
\begin{eqnarray*}
A_1  &=&\left\{ \left( u,v\right) \in C:\fl u  =r_1,\
\fl v \leq r_2\right\},\  \\
A_2 &=&\left\{ \left( u,v\right) \in C:\fl u \leq
r_1,\ \fl v =r_2\right\} .
\end{eqnarray*}%
Observe that \eqref{eq:ind0} implies that there are no fixed points of $N$ on $\partial U$. Therefore, the indices $\mathrm{ind}_{C}(N,U)$ and $\mathrm{ind}_{C}(N,C\setminus \overline{U})$ are well defined and their sum, by the additivity property of the index, is equal to one. Therefore, it
suffices to prove that $\mathrm{ind}_{C}(N,U)=0.$ Take $h=(R_1 \psi_1,R_2\psi _2)\in C$ and consider the homotopy $H:C\times \left[ 0,1\right] \rightarrow C$,
\begin{equation*}
H\left( u,v,t\right) :=th+(1-t)N(u,v).
\end{equation*}%
We claim that $H$ is fixed point free on $\partial U$. Since
\begin{eqnarray}\label{star}
\fl {R_{i}\psi _{i}}  &=&R_{i}\fl{ \psi _{i}}>r_{i}, \quad i=1,2.
\end{eqnarray}%
we have $\left( u,v\right) \neq h=H\left( u,v,1\right) $ for all $%
\left( u,v\right) \in \partial U.$ It remains to show that $H\left(
u,v,t\right) \neq \left( u,v\right) $ for $\left( u,v\right) \in \partial U$
and $t\in \left( 0,1\right) .$ Assume the contrary. Then there exists $%
(u,v)\in A_1 \cup A_2$ and $t\in (0,1)$ such that
\begin{equation}
(u,v)=th+(1-t)N(u,v).  \label{eq:homotopy not good}
\end{equation}%
Suppose that $(u,v)\in A_1 .$ Then,
\[N_1(u,v)\preceq |N_1(u,v)|\psi_1\preceq R_1\psi_1\] and
exploiting the first coordinate of the equation \eqref{eq:homotopy not good}, we obtain
\begin{equation}
u = N_1 (u,v)+ t(R_1 \psi _1 -N_1 (u,v))\succeq N_1(u,v).
\end{equation}%
Using the monotonicity of $\fl\cdot $ and \eqref{eq:ind0} we obtain
$r_1=\fl u\geq \fl{N_1(u,v)}>r_1$, which is impossible. Similarly, we derive a contradiction if $\left(u,v\right) \in A_2.$

By the homotopy invariance of the index we obtain $\mathrm{ind}_C(N,U)=\mathrm{ind}_C(h,U).$ From (\ref{star}) we have $h\not\in \overline{U}$, hence $\mathrm{ind}_C(N,U)=\mathrm{ind}_C(h,U)=0,$ as we wished.
\end{proof}

\begin{rem}\label{rem:inf>=}
From the proof we can deduce that if we change the assumption \eqref{eq:ind0} into
\begin{equation}\label{eq:ind0>=}
\inf\limits_{\substack{ (u,v)\in C  \\ \fl u  =r_1,\fl v  \leq
r_2}}\fl{ N_1 (u,v)} \geq r_1,\ \
\inf\limits_{\substack{ (u,v)\in C  \\ \fl u  \leq r_1,\fl v
=r_2}}\fl{ N_2(u,v)} \geq r_2,
\end{equation}
then we obtain at least one fixed point $(u,v)\in C$ with slightly weaker localization: $\fl u  \geq r_1 $ or $\fl v  \geq r_2$. The assumption \eqref{eq:ind0>=} permits the existence of fixed points of $N$ on $\partial U$. The assumption \eqref{eq:ind0} is more convenient when dealing with multiplicity results.
\end{rem}

\begin{rem}\label{rem:lowerbound_norm} We observe that, using the relation \eqref{eq:rel-fl-norm}, a lower bound for the solution in terms of the functional $\fl\cdot$ provides a lower bound for the norm of the solution, namely
\[\fl{u}> r_1 \implies |u|> \fl{\psi_1}^{-1} r_1,\quad \fl{v}> r_2 \implies |v|> \fl{\psi_2}^{-1} r_2.\]
\end{rem}
\begin{rem}\label{comp}
 The main differences between Theorem \ref{thm:existence 1a} and Theorem 2.17 of \cite{b:gi-mm-rp} consist in:
\begin{itemize}
\item The possibility of considering a positively homogeneous functional $\fl\cdot$ instead of a seminorm.
\item The assumption on the cone; in Theorem 2.17 of \cite{b:gi-mm-rp} it is needed the existence of  $\psi$ such that $u\leq |u|\psi_i$ for all $u\in E_i$, where $\leq$ is the order induced by the cone $K_i$, $i=1,2$. Here, instead, we can consider a semiorder.
\end{itemize}
Under the point of view of the applicability of our novel approach to parabolic problems, this is highlighted in the Remarks \ref{rem:choiceC0}
and \ref{rem:WhySuchHarnack}.
\end{rem}
The second Theorem is in the spirit of Theorem 2.9 and Remark 2.16 of \cite{b:gi-mm-rp}.
\begin{thm}\label{thm:existence 2a} 
Assume that there exist numbers $r_i,R_i$, $i=1,2$ with $0<r_i<\fl{\psi_i}R_i$ such that
\begin{equation}\label{eq:ind0'}
\inf\limits_{\substack{ (u,v)\in C  \\ \fl u  =r_1,\fl v  \geq
r_2}}\fl{ N_1 (u,v)} > r_1,\ \
\inf\limits_{\substack{ (u,v)\in C  \\ \fl u  \geq r_1,\fl v
=r_2}}\fl{ N_2(u,v)} > r_2,
\end{equation}
and
\begin{equation}\label{eq:ind1'}
\sup\limits_{(u,v)\in C}|N_{i}(u,v)|\leq R_{i}\ \ \ \ (i=1,2).
\end{equation}

Then $N$ has at least one fixed point $(u,v)\in K_1\times K_2$ such that $|u|\leq R_1,$
$|v|\leq R_2$ and  $\fl u  > r_1 $, $\fl v  > r_2.$
\end{thm}

\begin{proof}
The proof is similar to the proof of Theorem \ref{eq:ind0} and \cite[Theorem 2.9]{b:gi-mm-rp} and we only sketch it.
As before, the assumption \eqref{eq:ind1'} implies that $N(C)\subset C.$ Thus, $%
\mathrm{ind}_C(N,C)=1.$ In order to finish the proof, it is sufficient to
show that $\mathrm{ind}_C(N,V)=0,$ where
\begin{equation*}
V:=\left\{ \left( u,v\right) \in C:\fl{ u} <r_1 \ \text{or\
}\fl{ v} <r_2\right\} .
\end{equation*}%
We have $\partial V=B_1 \cup B_2,$ where
\begin{eqnarray*}
B_1  &=&\left\{ \left( u,v\right) \in C:\fl{ u} =r_1,\
\fl{ v} \geq r_2\right\}, \\
B_2 &=&\left\{ \left( u,v\right) \in C:\fl{ u} \geq
r_1,\ \fl{ v} =r_2\right\} .
\end{eqnarray*}%
By \eqref{eq:ind0'} we obtain $N$ has no fixed points on $\partial V$. Consider the same homotopy as in the proof of Theorem \ref{thm:existence 1a}, that is
\begin{equation*}
H\left( u,v,t\right) =th+(1-t)N(u,v),\ \text{ where }\ h=(R_1 \psi
_1,R_2\psi _2).
\end{equation*}%
As before we can prove that $H$ is fixed point free on $\partial V$. Therefore $\mathrm{ind}_C(N,V)=\mathrm{ind}_C(h,V)=0$, since, like in the previous proof, $h\notin \overline{V}.$
\end{proof}

The result, in the spirit of Lemma~4 of \cite{b:gipp-nonlin}, allows different types
of growth of the operators and is a modification of Theorem 2.4 of \cite{b:gi-mm-rp}.

\begin{thm}
\label{thm:existence or a} Assume that there exist numbers $r_i,R_i$ with $0<r_i<\fl{\psi_i}R_i$  such that
\begin{equation}\label{eq:ind1''}
\sup\limits_{(u,v)\in C}|N_{i}(u,v)|\leq R_{i}\quad  (i=1,2),
\end{equation}%
and
\begin{equation}\label{eq:ind0 or}
\inf_{(u,v)\in A} \fl{N_1(u,v)} \geq r_1\ {\text{ or }}\ \inf_{(u,v)\in A}
\fl{N_2(u,v)} \geq r_2,
\end{equation}
where $A$ is a subset of the set
\begin{equation*}
U=\mbox{$\left\{(u,v)\in C:\;\fl u < r_1,\ \fl v< r_2\right\}$}\subsetneq C.
\end{equation*}
Then $N$ has at least one fixed point $(u,v)\in K_1\times K_2$ such that $|u|\leq R_1,$
$|v|\leq R_2$ and $(u,v)\not\in A.$
\end{thm}

\begin{proof}
Since $N$ is a completely continuous mapping in the bounded closed convex
set $C,$ by Schauder's fixed point theorem, it possesses a fixed point $%
(u,v)\in C.$ We now show that the fixed point is not in $A.$ Suppose on the
contrary that $(u,v)=N(u,v)$ and $(u,v)\in A.$ Suppose that the first
inequality from \eqref{eq:ind0 or} is satisfied. Then
\begin{equation*}
r_1 >\fl u =\fl{ N_1 (u,v)} \geq r_1,
\end{equation*}%
which is impossible. Similarly we arrive at a contradiction, if the second
inequality from \eqref{eq:ind0 or} is satisfied.
\end{proof}

\subsection{Multiplicity results}
We present now some multiplicity results that are analogues of the results of Subsection 2.3 of \cite{b:gi-mm-rp}.

\begin{thm}\label{thm:existence 3sols a}
Assume that there exist numbers $\rho _{i}, r_{i},
R_{i}$ with
\begin{equation}
0<\fl{\psi_i}\rho _{i}<r_{i}<\fl{\psi_i} R_{i}\ \ \
(i=1,2),
\end{equation}%
such that
\begin{equation}\label{eq:ind0'i}
\inf\limits_{\substack{ (u,v)\in C  \\ \fl u  =r_1,\fl v  \geq
r_2}}\fl{ N_1 (u,v)} > r_1,\ \
\inf\limits_{\substack{ (u,v)\in C  \\ \fl u  \geq r_1,\fl v
=r_2}}\fl{ N_2(u,v)} > r_2,
\end{equation}

\begin{equation}\label{eq:ind1'''}
\sup\limits_{(u,v)\in C}|N_{i}(u,v)|\leq R_{i}\ \ (i=1,2),   \end{equation}%
and
\begin{equation}\label{eq:ind1-rho} N(u,v)\neq \lambda (u,v),\ \text{ for }\lambda \geq 1\text{ and }(\left| u\right| =\rho _1,\ \left| v\right| \leq \rho _2\text{ or } \left| u\right| \leq \rho _1,\ \left| v\right| =\rho _2).
 \end{equation}
Then $N$ has at least three fixed points $(u_{i},v_{i})\in C$ $(
i=1,2,3) $ with
\begin{align*}
| u_1 | &<\rho _1,\ | v_1 |<\rho _2\ ( \text{possibly zero solution}); \\
\fl{u_2} &<r_1  \text{ or }  \fl{v_2} <r_2;\ | u_2| >\rho _1 \ \text{or }|v_2| >\rho _2\ ( \text{possibly one solution component zero}) ; \\
\fl{ u_{3}} &>r_1,\ \fl{v_{3}}>r_2\  (\text{both solution components nonzero}).
\end{align*}
\end{thm}
\begin{proof}
Let $U,V$ be as in the proof of Theorems \ref{thm:existence 1a} and \ref{thm:existence 2a}. Strict inequalities
in (\ref{eq:ind0'i}) guarantee that $N$ is fixed point free on $\partial V.$
According to the proof of Theorem \ref{thm:existence 2a} we have $\mathrm{ind}%
_{C}(N,C)=1,$ $\mathrm{ind}_{C}(N,V)=0$ and therefore by the additivity
property, $\mathrm{ind}_{C}(N,C\setminus \overline{V})=1.$ Let%
\begin{equation*}
W:=\{ ( u,v) \in C:| u| <\rho _1,\
| v| <\rho _2\} .\text{ }
\end{equation*}%
From \eqref{eq:rel-fl-norm}, for every $( u,v) \in \overline{W},$ we have
\begin{equation*}
\fl{u}\leq \fl{\psi_1} |u| \leq \fl{\psi_1}\rho_1 <r_1
\end{equation*}%
and, similarly, $\fl v <r_2.$ Hence $(u,v)
\in U,$ which proves that $\overline{W}\subset U \subset V.$ Condition (\ref%
{eq:ind1-rho}) shows that $N$ is homotopic with zero on $W.$ Thus $\mathrm{%
ind}_{C}(N,W)=\mathrm{ind}_{C}(0,W)=1.$ Then $\mathrm{ind}_{C}(N,V\setminus
\overline{W})=0-1=-1.$ Consequently, there exist at least three fixed points
of $N,$ in $W,\ V\setminus \overline{W}$ and $C\setminus \overline{V}.$
\end{proof}

If we assume the following estimates of $\fl{N_{i}(u,v)}$:

\begin{equation}\label{eq:ind0''}
\inf\limits_{\substack{ (u,v)\in C  \\ \fl u  =r_1 }}\fl{ N_1 (u,v)} > r_1,\ \
\inf\limits_{\substack{ (u,v)\in C  \\ \fl v =r_2}}\fl{ N_2(u,v)} > r_2,
\end{equation}
then we can obtain a more precise localization for the solution $\left(
u_2,v_2\right) $ in Theorem \ref{thm:existence 3sols a}, the Figure \ref{fig:3sols} (analogous to Figure 1 of \cite{b:gi-mm-rp}) illustrates this fact.

\begin{figure}[h]
\centering
\includegraphics[width=\textwidth]{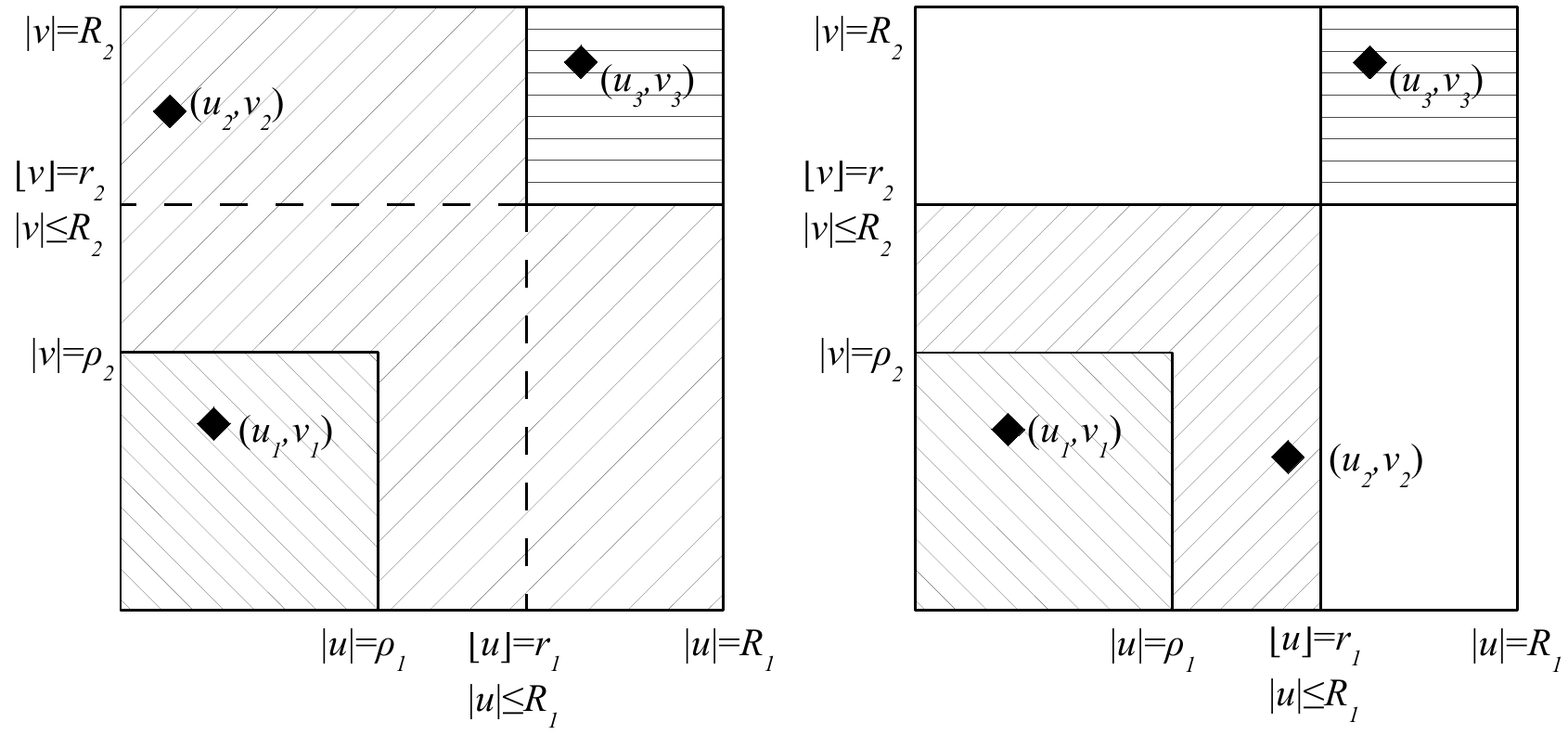}
\caption{Localization of the three solutions $(u_{i},v_{i})$ from Theorem
\protect\ref{thm:existence 3sols a} (on the left) and Theorem \protect\ref{thm:existence 3sols a'}
(on the right).}
\label{fig:3sols}
\end{figure}

\begin{thm}
\label{thm:existence 3sols a'} Suppose that all the assumptions of Theorem \emph{%
\ref{thm:existence 3sols a}} are satisfied with the condition \emph{(\ref{eq:ind0'i})}
replaced by \emph{(\ref{eq:ind0''})}. Then $N$ has at least three fixed points $%
(u_{i},v_{i})\in C$ $\left( i=1,2,3\right) $ with%
\begin{eqnarray*}
| u_1 | &<&\rho _1,\ \ | v_1 |
<\rho _2\ \ ( \text{possibly zero solution}); \\
\fl{u_2} &<&r_1,\ \ \fl{v_2}
<r_2;\ | u_2| >\rho _1 \ \text{or }|
v_2| >\rho _2\ \ ( \text{possibly one solution component
zero}) ; \\
\fl{ u_{3}} &>&r_1,\ \fl{v_{3}}>r_2\ \
( \text{both solution components nonzero}).
\end{eqnarray*}

\end{thm}

\begin{proof}
The assumption (\ref{eq:ind0''}) implies both (\ref{eq:ind0}) and (\ref{eq:ind0'}) and that
there are no fixed points of $N$ on $\partial U$ and $\partial V.$ Hence, as
in the proofs of Theorems \ref{thm:existence 1a} and \ref{thm:existence 2a}, the
indices $\mathrm{ind}_{C}(N,U)$ and $\mathrm{ind}_{C}(N,V)$ are well defined
and equal $0.$ An analysis similar to that in the proof of Theorem \ref%
{thm:existence 3sols a} shows that
\begin{equation*}
\mathrm{ind}_{C}(N,W)=1,\ \mathrm{ind}_{C}(N,U\setminus \cl{W})=-1,\
\mathrm{ind}_{C}(N,C\setminus \cl{V})=1,
\end{equation*}%
which completes the proof.
\end{proof}

In order to ensure that the solution $(u_1,v_1 )$ from the theorems above
is nonzero, and thereby to obtain three \emph{nonzero} solutions, we use
some additional assumptions on $N$.

\begin{thm}\label{thm:existence 3sols a''}
Assume that all the conditions of Theorem {\ref{thm:existence 3sols a}} or Theorem {\ref{thm:existence 3sols a'}} are satisfied.
Consider $0<\varrho_i<\fl{\psi_i}\rho _i,
$ $i=1,2.$

\emph{(i)} If $N_1(0,0)\neq 0$ or $N_2(0,0)\neq 0$, then  the solution $(u_1,v_1 )$ from Theorem {\ref{thm:existence 3sols a}} or {\ref{thm:existence 3sols a'}} is nonzero.

\emph{(ii)} If
\begin{equation}
\inf\limits_{\substack{ (u,v)\in K,|u|\leq \rho _1,|v|\leq \rho _2 \\
\fl u =\varrho _1,\fl v \geq \varrho _2}}\fl{ N_1 (u,v)} \geq  \varrho_1,\ \
\inf\limits_{\substack{ (u,v)\in K,|u|\leq \rho _1,|v|\leq \rho _2 \\
\fl u \geq \varrho _1,\fl{v} =\varrho _2}}\fl{N_2(u,v)} \geq \varrho_2
\label{6varrho}
\end{equation}%
and
\begin{equation*}
|N_i (u,v)|\leq \rho _i \ \ \text{for\ \ }|u|\leq \rho _1,\ |v|\leq \rho
_2\ \ \ \left( i=1,2\right),
\end{equation*}%
then we can assume that the solution $(u_1,v_1 )$ from Theorem {\ref{thm:existence 3sols a}} or Theorem {\ref{thm:existence 3sols a'}} satisfies $\fl{u_1} \geq \varrho _1 $ and $\fl{v_1}\geq \varrho _2$;

\emph{(iii)} if
\begin{equation}
\inf\limits_{\substack{ (u,v)\in K,|u|\leq \tilde{\rho}_1,|v|\leq \tilde{
\rho}_2  \\ \fl u < \varrho_1,\fl v < \varrho _2}}
\fl{N_1 (u,v)} \geq \varrho _1 \ {\text{ or }}\ \inf\limits
_{\substack{ (u,v)\in K,|u|\leq \tilde{\rho}_1,|v|\leq \tilde{\rho}_2
\\ \fl u < \varrho _1,\fl v < \varrho _2}}\fl{N_2(u,v)} \geq \varrho _2  \label{6 or varrho}
\end{equation}%
for some$\ \tilde{\rho}_1 \leq \rho _1,\ \tilde{\rho}_2\leq \rho _2,$
then we can assume that the solution $(u_1,v_1 )$ from Theorem {\ref{thm:existence 3sols a}} or Theorem {\ref{thm:existence 3sols a'}} satisfies $\fl{u_1} \geq \varrho_1 $ or $\fl{v_1} \geq \varrho _2$ or $|u_1 |>\tilde{\rho}_1 $ or $|v_1 |>\tilde{\rho}_2.$
\end{thm}

\begin{proof}
(i) The assumption impies that $(0,0)$ is not a fixed point.

(ii) The inequality follows from Theorem \ref{thm:existence 2a} applied in the
case of $r_i:=\varrho_i $ and $R_i :=\rho_i.$

(iii) From Theorem \ref{thm:existence or a} applied in the case of $r_i :=\varrho _i,$ $R_i :=\rho _i $ and
\[A=\set{(u,v)}{\fl u < \varrho_1, \fl v < \varrho_2, |u|\leq \tilde\rho_1, |v|\leq \tilde\rho_2},\]
we obtain there are no fixed points of $N$ in $A$, which ends the proof.
\end{proof}

The next Remark illustrates how Theorem \ref{thm:existence 2a} can be used to
prove the existence of more nontrivial solutions.

\begin{rem}\label{rem:many-sols}
If $N$ satisfies the conditions of Theorem \ref{thm:existence 2a} for all pairs

\begin{equation*}
0<r_i^j<\fl{\psi_i}R_i^j \ \text{for }\ i=1,2,\ j=1,2,\ldots,n,\end{equation*}
satisfying
\begin{equation*}
\fl{\psi_i}R_i^j<r_i^{j+1}\ \text{ for }\ i=1,2,\ j=1,2,\ldots,n-1,
\end{equation*}
then $N$ possesses at least $n$ nontrivial solutions $(u_j,v_j)$ with
\begin{equation*}
|u_j |\leq R_1 ^j,\ |v_j |\leq R_2^j,\ \fl{u_j} >
r_1^j,\ \fl{v_j} > r_2^j.
\end{equation*}%
Moreover, if \eqref{eq:ind1'} holds with the strict inequality, i.e. if
\[\sup\limits_{(u,v)\in K,|u|\leq R_1 ^j,|v|\leq R_2^j }|N_i (u,v)|<R_i ^j \quad(i=1,2),
\]
hold, then we have $n-1$ additional solutions $(\bar{u}_j,\bar{v}_j ),$ $%
j=1,\ldots,n-1$ such that
\begin{equation*}
|\bar u_j|<R_1^{j+1},\ |\bar v_j|<R_2^{j+1};\ |\bar u_j|>R_1^j\text{ or }|\bar v_j|>R_2^j;\ \fl{\bar u_j}
<r_1^{j+1}\text{ or }\fl{ \bar v_j} <r_2^{j+1}.
\end{equation*}%
The first conclusion follows from Theorem \ref{thm:existence 2a} applied $n$ times, whereas the second follows from Theorem \ref{thm:existence 3sols a} applied $
n-1$ times.
\end{rem}

\begin{rem}
We stress that the abstract results obtained in this section can be applied to the case of one equation. Furthermore, our results can be generalized to the case of systems of more than two equations. The idea is to consider the product space $E=\Pi_{i=1}^nE_i$ of the Banach spaces $E_i,$ endowed with the norms $|\cdot|_i,$ functionals $\fl\cdot_i $, and the pairs of cones and wedges $K_i \subset G_i\subset E_i $ such that \eqref{eq:fl monot}, \eqref{eq: prec-max} are satisfied for $i=1,2,...,n.$ In
 this setting we are interested in the existence and localization of fixed points of a given operator $N\colon K\rightarrow K,$ where $K=\Pi _{i=1}^n K_i .$ For example, let us consider the sets
\begin{equation*}
C=\mbox{$\left\{u\in K:\;|u_1|_1\leq R_1,\ldots,|u_n|_n\leq R_n\right\}$},\
U=\mbox{$\left\{u\in C:\;\fl{u_1}_1 < r_1,\ldots,\fl{u_n}_n < r_n\right\}$}
\end{equation*}%
for given radii $r_i,R_i>0 $ with $r_i <\fl{\psi_i}_i R_i,\ $ $i=1,...,n.$ If
\begin{equation*}
\sup_{u\in C}|N_i (u)|_i \leq R_i,\ \ i=1,2,...,n
\end{equation*}%
and
\begin{equation*}
\inf_{\substack{ u\in \overline{U}  \\ \fl{ u_i}_i =r_i }}%
\fl{N_i (u)}_i> r_i,\quad
i=1,2,...,n,
\end{equation*}%
then $N$ has at least one fixed point in $C\setminus \cl U.$

As a consequence, results analogous to ones obtained later in Section \ref%
{sect:appl}, can be established for systems with more than two differential
equations.
\end{rem}

\section{The system of parabolic equations}\label{sect:appl}
Let $\Omega\subset \R^m$ be a bounded domain that is Dirichlet regular.
This class of domains is rather large; for example, if the boundary of $\Omega$ is Lipschitz continuous, then $\Omega$ is Dirichlet regular (see \cite[Chapter II, Section 4, Proposition 4]{b:DL}).

Let us take
\[E=C_0(\Omega)=\set{u\in C(\cl{\Omega})}{u|\partial\Omega=0},\ E_+=\set{u\in E}{u(x)\geq 0\text{ for all }x\in\Omega}.\]
Let us also consider the space $\E=C(0,\tm,E)$, $\tm>0$ and its cone of nonnegative functions $\E_+=C(0,\tm,E_+)$. The spaces $E$ and $\E$ are endowed with the uniform norms, that is
\[|u|=\max\set{|u(x)|}{x\in \cl\Omega},\ u\in E, \] and
\[|u|=\max\set{|u(t)|}{0\leq t\leq \tm},\ u\in \E.\]

Let $D\subset\subset  \Omega$ be any open non-empty subset. Put
\[{G} = \set{u\in \E}{u(0,x)\geq 0\ \text{ for all }x\in D}.\]
The set $G$ is a wedge generating the semiorder  $\preceq$. By $\leq$ we denote the order induced by the cone $\E_+$. The symbol $\leq$ will also be used to denote  the order on $E$ induced by $E_+$ and the natural order on $H=L^2(\Omega)$ (that is $u\leq v$ if $u(x)\leq v(x)$ for a.a. $x\in\Omega$).

Given a function $u\in H$ we set
$$
\fl u=\einf_{x\in D}|u(x)|,
$$
(in particular
$\fl u=\inf_{x\in D} |u(x)|$ for $u\in E$) and futhermore, with abuse of notation, by the same symbol we denote the value $\fl u=\fl{u(0)}$ for $u\in\E$.

The following monotonicity and continuity conditions of the functional $\fl\cdot$ are satisfied:
\[\fl u\leq \fl v\text{ if }u,v\in\E_+\text{ and }u\preceq v\] and
\[|\fl u-\fl v|\leq |u-v|\text{ for all }u,v\in\E \quad\left(\text{therefore} \fl u\leq |u|\text{ for all }u\in\E\right).\]

Consider the function $\psi(t)=\varphi$ for all $t\in[0,\tm ]$, where $\varphi\in E$ satisfies the following conditions: $\varphi|D\equiv 1$ and $0\leq \varphi\leq 1$ in $\Omega$. Then $|\psi|=1$ and $u\preceq|u|\psi$ for every $u\in\E_+$.

Let us consider the Laplace operator $\Delta$ defined on the domain $D(\Delta) =\set{u\in E}{\Delta u\in E}$. We briefly recall some known facts regarding the Laplace operator and the semigroup generated by it.

\begin{lem} The operator $\Delta$ is a generator of an analytic (immediately) compact $C_0$-semigroup of contractions $\set{S(t)}{t\geq 0}$ on $E$. Moreover, the operators $S(t)$ are positive, i.e. $S(t)u\geq 0$ for $u\geq 0$.
\end{lem}
\begin{proof} See~\cite[Section 6.1]{b:ABHN}.
\end{proof}

We shall also consider the space $H=L^2(\Omega)$ and the Laplacian $\Delta_2$ on $H$ with Dirichlet boundary condition. Denote by $S_2\colon [0,\infty)\to B(H)$ the semigroup generated by $\Delta_2$ and by $i\colon E\to H$ the natural embedding.
\begin{prop}\label{prop:iS=S2i} $i(S(t)u)=S_2(t)i(u)$ for all $u\in E$.
\end{prop}
\begin{proof} The proof uses the Post-Widder inversion formula for $C_0$-semigroups (see Corollary 3.3.6 in \cite{b:ABHN}).
\end{proof}

\begin{defin}\label{def:mild-solution}
For $\xi\in E$, and $f\in \E$, we say that the function \[u(t)=p(\xi,f)(t):=S(t)\xi+\int_0^t S(t-s)f(s)ds\] is a {\em mild solution} of the problem $u'-\Delta u=f$ on $(0,\tm)\times\Omega$ with $u(0)=\xi$.
\end{defin}

It is worth pointing out the following regularity result of the mild solutions, this can be proved using standard techniques, see for example \cite[Theorem 8.2.1]{b:Vrabie}.

\begin{prop}\label{prop:strong sol}
Let $u=p(\xi,f)$ for $\xi\in E$, $f\in\E$, that is $u$ is a mild solution of the problem $u'-\Delta u=f$, $u(0)=\xi$. Then
\begin{enumerate}[\upshape (a)]
\item $u$ is a strong solution of that problem in the space $L^2(\Omega)$. Precisely, $u\colon [0,\tm ]\to L^2(\Omega)$ is absolutely continuous, $u'\in L^1(0,\tm,L^2(\Omega))$, $u(t)\in D(\Delta_2)$ and $u'(t)=\Delta_2 u(t)+f(t)$ for almost all $t\in(0,\tm )$.
\item $u$ is a weak solution of the equation $u'-\Delta_2u=f$ in the sense that the weak spatial derivative $\nabla u(t,x)$ and weak time derivative $u_t(t,x)$ exist on $(0,\tm)\times\Omega$ and

\begin{equation}\label{eq:weak solution}
\int_0^{t_0}\int_\Omega\nabla u(t,x)\nabla \phi(t,x)-u\frac{\partial}{\partial t}\phi(t,x)dxdt=\int_0^{t_0}\int_\Omega f(t)\phi(t,x)dt
\end{equation}
for all $\phi\in C^\infty_0((0,\tm )\times\Omega)$.

\end{enumerate}
\end{prop}

We make use of the following result.
\begin{prop}\label{prop:properties u}
Assume that $u_0\in E$ is a nonnegative nonzero function. Define $u(t,x)=(S(t)u_0)(x)$. Then
\begin{enumerate}[\upshape (i)]
	\item $u\in C^\infty((0,\infty)\times\Omega)\cap C(\R_+\times \cl\Omega)$
	\item $u_t=\Delta u$ on $(0,\infty)\times\Omega$
	\item $u(0,x)=u_0(x)$ for $x\in\Omega$
	\item $u(t,x)>0$  for $t>0$ and $x\in\Omega$.
\end{enumerate}
\end{prop}
\begin{proof}
The conclusions (i)-(iii) follow from Proposition 2.6 of \cite{b:ABHN}, while (iv) follows from the parabolic Harnack inequality, see for example Theorem 7.1.10 of \cite{b:Evans}.
\end{proof}
\begin{rem}\label{rem:choiceC0}%\label{rem:nopsi} 
It seems worth discussing the choice of the space $C_0(\Omega)$. In the recent paper \cite{b:gi-mm-rp}, where an elliptic system was discussed, the space $L^\infty(\Omega)$ was considered. Unfortunately, the Laplacian $\Delta$ fails to generate a $C_0$-semigroup on $L^\infty(\Omega)$. Moreover, although $\Delta$ generates semigroups on $L^p(\Omega)$, these spaces are somewhat inappropriate to obtain the localization of solutions with our approach. 
Note also that in the space $\E=C(0,\tm,C_0(\Omega))$ there is no element $\bar\psi$ that  $|\bar\psi|=1$ and $u\leq|u|\bar\psi$ for every $u\in\E_+$. This fact prevented us from using the abstract setting from \cite{b:gi-mm-rp} and is the main reason for considering the wedges $G_i$, $i\in\{1,2\}$, and the semiorders $\preceq$ in Section \ref{sect:abstr}.
\end{rem}
We make use of the following Lemma.
\begin{lem}\label{lem:m_eta>0} Let $0\neq \eta\in E=C_0(\Omega)$ and let $\mu$ be a Borel measure on $\R$.
\begin{enumerate}[\upshape (a)]
	\item  If $0<t_0<t_1$, then we have
	 	\[m_\eta(t_0,t_1):=\min\set{(S(\tau)\eta)(x)}{x\in \cl D,\ \tau\in[t_0,t_1]}>0.\]
	 \item If $0\leq t_0<t_1$ and $\mu([t_0,t_1])>0$, then we have
	$\fl{\int_{t_0}^{t_1}S(\tau)\eta d\mu(\tau)}>0$.
\end{enumerate}
\end{lem}
\begin{proof}
(a) The conclusion follows from Proposition \ref{prop:properties u} 
and compactness of $[t_0,t_1]\times \cl D$.

(b) We can assume that $t_0>0$. Then
\[\int_{t_0}^{t_1}(S(\tau)\eta)(x) d\mu(\tau)\geq \mu([t_0,t_1])m_\eta(t_0,t_1)>0\] for $x\in D$, which is the desired conclusion.
\end{proof}

Define
\begin{equation}\label{eq:def m} m(t_0,t_1):=\einf_{x\in D,\ t\in[t_0,t_1]}\left(S_2(t)\chi_D\right)(x).\end{equation}

\begin{cor}\label{cor:m>0}
\begin{enumerate}[\upshape (a)]
	\item If $0<t_0<t_1$, then $m(t_0,t_1)>0$.
	\item If $0\leq t_0\leq t_1$ and $\mu([t_0,t_1])>0$, then
$\fl{\int_{t_0}^{t_1}S_2(\tau)\chi_Dd\mu(\tau)}>0$.
\end{enumerate}
\end{cor}
It is possible that $m(0,t_1)>0$ (see Example \ref{exa:ex1}). 
\begin{proof}
Let $\eta\in E=C_0(\Omega)$ be any nonzero function with $0\leq \eta\leq \chi_D$.  The conclusion is implied by Lemma \ref{lem:m_eta>0} and the inequality
$S_2(t)\chi_D\geq S_2(t)\eta= S(t)\eta,$ which follows from the positiveness of $S_2$ and Proposition \ref{prop:iS=S2i}.
\end{proof}
Define the cones
\[K_1=K_2=\set{u\in \E_+}{u(t)\geq S(t)u(0) \ \text{for all}\ t\in[0,\tm]}.\]

Let us fix $0\leq t_0<t_1\leq\tm$ satisfying $\mu_1([t_0,t_1])>0$, $ \mu_1([t_0,t_1])>0$, where $\mu_i$ are the measures from the definition of $\alpha$ and $\beta$. Put $m:=m(t_0,t_1)$.

\begin{rem}\label{rem:Harnack}
Take $u\in K$ and set $r:=\fl u=\inf_{x\in D}u(0,x)$. Observe that
$u(0)\geq r\chi_D$ and therefore we have
\[u(t)\geq S(t)u(0)=S_2(t)u(0)\geq rS_2(t)\chi_D\geq rm\chi_D\ \text{ for }t\in[t_0,t_1].\]
As a consequence, we obtain the estimate
\begin{equation}\label{eq:Harnack}
u\geq m\fl u\chi_{[t_0,t_1]\times D}\ \text{ for all }u\in K,
\end{equation}
which can be called \textit{weak Harnack-type inequality}, a counterpart of the inequality (3.4) of \cite{b:gi-mm-rp}.
\end{rem}
We now turn back our attention to the parabolic system
\begin{equation}\label{eq:parabolic}
\begin{cases}
u_t-\Delta u = f(t,x,u,v) & (t,x)\in Q:=(0,\tm)\times\Omega,\\
v_t-\Delta v = g(t,x,u,v) & (t,x)\in Q:=(0,\tm)\times\Omega,\\
u(t,x)=v(t,x)=0 & (t,x)\in (0,\tm)\times\partial\Omega,\\
u(0,\cdot)=\alpha(u,v),\\
v(0,\cdot)=\beta(u,v).
\end{cases}
\end{equation}

Here,  $f,g\colon (0,\tm)\times\Omega\times\R_+\times\R_+\to\R_+$ and $\alpha,\beta\colon\E\times\E\to E$  are continuous functions. In what follows we shall  identify $u\colon[0,\tm]\times\cl\Omega\to\R$ with $u\colon [0,\tm ]\to C(\cl\Omega)$ via the formula $u(t)(x)=u(t,x)$.

Define \[F(u,v)(t)(x)=f(t,x,u(t)(x),v(t)(x)),\quad G(u,v)(t)(x)=g(t,x,u(t)(x),v(t)(x))).\] Under the following assumption:
\begin{equation}\label{eq:assumpt fg}
f(t,x,0,0)=g(t,x,0,0)=0\ \text{ for }x\in\partial\Omega,
\end{equation}
the operators $F,G\colon\E\times\E\to\E$ are continuous and bounded (map bounded sets into bounded ones).

Therefore, the mild solution of the problem \eqref{eq:parabolic} is a fixed point of the vector valued operator $M$ defined as
\[M_1(u,v)=\bar S(\alpha(u,v))+\hat S(F(u,v)),\ M_2(u,v)=\bar S(\beta(u,v))+\hat S(G(u,v)),\]
where
\[\bar S\colon E\to \E,\ \bar S(u_0)(t)=S(t)u_0,\ u_0\in E,\]
\[\hat S\colon \E\to\E,\ \hat S(f)(t)=\int_0^tS(t-\tau)f(\tau)d\tau, f\in\E.\]

\begin{prop}\label{prop:M-compact} The operators $M_1,M_2$ are completely continuous if and only if $\alpha$ and $\beta$ are completely continuous.
\end{prop}
\begin{proof}
If $M_1$, $M_2$ are completely continuous, then $e_0\circ M_1=\alpha, e_0\circ M_2=\beta$ are completely continuous, where $e_0(f)=f(0)$ for $f\in\E$.

Now, assume that the operators $\alpha,\beta$ are completely continuous. Since the operator $\bar S$ is continuous, then the operator $\bar S\circ \alpha$ is completely continuous. We shall demonstrate that $\hat S$ is completely continuous. In order to do this we utilize a version of the Ascoli-Arzel\`{a} Theorem tailored for the space $\E$, see for example Theorem A.2.1 of \cite{b:Vrabie}.

First, denote by $\Xi$ the upper bound of the norms of $\|S(t)\|_{B(E)}$ for $t\in[0,\tm ]$. Because $S$ is immediate norm continuous, for any $\varepsilon>0$ there exists a number $\delta(\varepsilon)>0$ such that $\|S(t)-S(s)\|_{B(E)}<\varepsilon$ if $0<\varepsilon\leq t\leq s\leq t+\delta(\varepsilon)$.

Fix $R,\varepsilon>0$ and let $D=\set{f\in \E}{|f|<R}$. For a fixed $t>0$ and $f\in D$ let $\eta=\min\{t,\varepsilon\}$. We can present $\hat S(f)(t)$ as a sum $\hat S(t)=x+y$, where
\[x=S(\eta)\int_0^{t-\eta}S(t-s-\eta)f(s)ds,\quad y=\int_{t-\eta}^tS(t-s)f(s)ds.\]
It is straightforward to show that $x\in \hat C:=S(\eta)(B(0,t\Xi R))$ and that $y\in B(0,\varepsilon \Xi R)$. Because $\hat C$ is relatively compact and $\varepsilon$ is arbitrary, we obtain the set $\set{\hat S(f)(t)}{f\in D}$ is relatively compact for all $t\in[0,\tm ]$.

Now we shall prove the equicontinuity of the family $\set{\hat S(f)}{f\in D}$. In order to do it let us fix $\varepsilon>0$, $t,s\in [0,\tm ]$ and $f\in D$. Without loss of generality we can assume that $t\leq s$. Let us put $\eta=\min\{t,\varepsilon\}$ (the case $\eta=0$ appears if $t=0$). Then $\hat S(f)(s)-\hat S(f)(t)=x+y$, where
\[x=\int_0^{t}(S(s-\tau)f(\tau)-S(t-\tau)f(\tau))d\tau,\ y=\int_t^sS(s-\tau)f(\tau)d\tau.\]
\[ |x| \leq R\int_0^t\|S(s-t+\tau)-S(\tau)\|_{B(E)}d\tau\leq 2\Xi R\varepsilon+\int_\eta^t\|S(s-t+\tau)-S(\tau)\|_{B(E)}d\tau.\]
\[ |y | \leq |s-t|\Xi R.\]
Therefore $|\hat S(f)(s)-\hat S(f)(t)|\leq (3\Xi R+\tm)\varepsilon$ if $|s-t|\leq \min\{\delta(\varepsilon),\varepsilon\}$.
This proves the uniform equicontinuity of the family $\set{\hat S(f)}{f\in D}$ and finishes the proof of the complete continuity of $\hat S$. Now, the complete continuity of $M_1=\bar S\circ \alpha+\hat S\circ F$ is clear. Similarly we can prove the complete continuity of $M_2$.
\end{proof}

In order to use fixed point index for compact operators and, at the same time, to avoid assuming the compactness of $\alpha$ and $\beta$, we consider the operator $N=(N_1,N_2)$, defined by the formula
\[N_1(u,v)=\bar S\left(\alpha\left(\bar S(u(0))+\hat S(F(u,v))\:,\: \bar S(v(0))+\hat S(G(u,v))\right)\right)+\hat S(F(u,v)),\]
\[N_2(u,v)=\bar S\left(\beta\left(\bar S(u(0))+\hat S(F(u,v))\:,\:\bar S(v(0))+\hat S(G(u,v))\right)\right)+\hat S(G(u,v)).\]

\begin{prop}\label{prop:fixed-sets} 
The sets of fixed points of the operators $M$ and $N$ coincide.
\end{prop}
\begin{proof}
Note that
\begin{equation}\label{eq:t1}
N_1(u,v)=\bar S(u_0)+\hat S(F(u,v))\text{ and }N_2(u,v)=\bar S(v_0)+\hat S(G(u,v)),
\end{equation} where
$u_0=\alpha(\bar u,\bar v)$, $v_0=\beta(\bar u,\bar v)$ and
\begin{equation}\label{eq:t2}
\bar u=\bar S(u(0))+\hat S(F(u,v)),\quad \bar v=\bar S(v(0))+\hat S(G(u,v)).
\end{equation}
From \eqref{eq:t1} and the properties of $\bar S$ and $\hat S$ we have $N_1(u,v)(0)=u_0$ and $N_2(u,v)(0)=v_0$. Therefore, if $N(u,v)=(u,v)$, then
$\bar u=u$ and $\bar v=v$ and, consequently, $u_0=\alpha(u,v)$ and $v_0=\beta(u,v)$. By \eqref{eq:t1} we arrive at $M(u,v)=(u,v)$.

Conversely, if $M(u,v)=(u,v)$, then
\begin{equation}\label{eq:t3}u=\bar S(\alpha(u,v))+\hat S(F(u,v)),\ v=\bar S(\beta(u,v))+\hat S(G(u,v))\end{equation} and $u(0)=\alpha(u,v)$, $v(0)=\beta(u,v)$. Therefore we have
\[u=\bar S(u(0))+\hat S(F(u,v))\text{ and }v=\bar S(v(0))+\hat S(F(u,v)).\] Plugging this into \eqref{eq:t3} we obtain $N(u,v)=(u,v)$.
\end{proof}

From the proof it follows in particular, that
\begin{equation}\label{eq:N,N(0)}
\begin{array}{c}
N_1(u,v)=\bar S(N_1(u,v)(0))+\hat S(F(u,v)),\text{ and }\\ N_2(u,v)=\bar S(N_2(u,v)(0))+\hat S(G(u,v)),
\end{array}
\end{equation}
which 
yields $N(\E_+\times\E_+)\subset K_1\times K_2=K$.

From Proposition~\ref{prop:M-compact} we know that a necessary condition for the operator $M$ to be completely continuous is the complete continuity of $\alpha$ and $\beta$. In the case of the operator $N$ we can weaken the assumptions on $\alpha$ and $\beta$.

\begin{prop}\label{prop:N-compact} The operator $N$ is completely continuous if the images $\alpha(U_1\times U_2)$ and $\beta(U_1\times U_2)$ are relatively compact for all bounded sets $U_1,U_2\subset \E$ that satisfy the following property:
\begin{equation}\label{eq:almost-compact-sets}
\text{For all}\ \varepsilon>0\ \text{the set}\ \set{u(t)}{u\in U_i,\ t\in[\varepsilon,\tm]}\ \text{is relatively compact},\ i=1,2.	
\end{equation}
\end{prop}

\begin{proof} We prove, without loss of generality,  the complete continuity of $N_1$. Let $\bar\alpha(u,v)=\alpha(\bar u,\bar v)$, where $\bar u(u,v)$ and $\bar v(u,v)$ are defined by \eqref{eq:t2}. From \eqref{eq:t1} it follows that $N_1(u,v)=\bar S(\bar\alpha(u,v))+\hat S(F(u,v))$. By Proposition \ref{prop:M-compact}, it suffices to show that $\bar\alpha$ is completely continuous. This will be done if we demonstrate, that $\bar u(U\times U)$ and $\bar v(U\times U)$ satisfy \eqref{eq:almost-compact-sets} for a given bounded set $U\subset {\mathcal E}$.

Let $\varepsilon>0$ be given. Put $R=\sup\set{|u|}{u\in U}$. For $u,v\in U$ and $t\geq \varepsilon$ we obtain
\[\bar S(u(0))(t)=S(\varepsilon)S(t-\varepsilon)u(0)\in S(\varepsilon)D(0,R)=:{\mathcal C}\] and the set ${\mathcal C}$ is relatively compact.
Moreover, the proof of Proposition \ref{prop:M-compact} shows that the set  $\hat S( F(U\times U))$  is relatively compact in $\mathcal E$. By the standard arguments, utilizing the compactness of $[0,\tm]$, one can show that the set \[\set{\hat S(F(u,v))(t)}{u,v\in U,\ t\in[0,\tm]}\] is totally bounded (and therefore relatively compact) in $E$. This shows that $\bar u(U\times U)$ satisfies the condition \eqref{eq:almost-compact-sets}. Similarly we can verify this condition for the set $\bar v(U\times U)$. \end{proof}

\begin{exa}\label{ex:alpha=u(t_0)} Let $\alpha(u,v)=u(t_0)$, $\beta(u,v)=v(t_0)$, where $0\leq t_0\leq \tm$. Then $\alpha,\beta$ satisfy the condition from Proposition \ref{prop:N-compact} if and only if $t_0>0$. Indeed, let $t_0=0$ and $U_1=U_2=\set{S(\cdot)u}{|u|\leq 1}$. Then the condition \eqref{eq:almost-compact-sets} is satisfied, but the set
\[\alpha(U_1\times U_2)=\set{u(0)}{u\in U_1}=\set{u\in E}{|u|\leq 1}\] is not compact.

Conversely, if $t_0>0$ and sets $U_1,U_2$ satisfy the condition \eqref{eq:almost-compact-sets}, then $\alpha(U_1\times U_2)=\set{u(t_0)}{u\in U_1}$ is compact from \eqref{eq:almost-compact-sets}.
\end{exa}

\begin{prop}\label{ex:alpha-integral}
Let $\alpha, \beta$ be as in \eqref{init-con}. Then $\alpha,\beta$ satisfy the assumptions of Proposition~\ref{prop:N-compact}.
\end{prop}
\begin{proof}
Indeed, let us consider the sets $U_1,U_2\subset B(0,R)$ satisfying the condition \eqref{eq:almost-compact-sets} and let $\varepsilon>0$. From the uniform continuity of $G_1$ on $[0,\omega\cdot d]$, where
\[d=\max\set{g_1(u,v)}{0\leq u,v\leq R},\  \omega=\mu_1([0,\tm]),\] there exists $\delta>0$ such that $|G_1(p)-G_1(q)|< \varepsilon$ if $0\leq p,q\leq \omega\cdot d$, $|p-q|<\delta\cdot d$.

Since $\mu_1(\{0\})=0$, there exists $\sigma>0$ such that $\delta_0:=\mu_1([0,\sigma))<\delta$.

Put
\[U_i^\sigma := \set{u(t)}{u\in U_i,\ t\in[\sigma,\tm]},\ i=1,2\text{ and }
\Gamma^\sigma := \cl{\mathrm{conv}}\: g_1\!\left(U_1^\sigma\times U_2^\sigma\right).\]

From the Mazur Theorem, which states that the closed convex hull of a compact subset in a Banach space is compact, the set $\Gamma^\sigma$ is compact. Therefore,
\[\int_\sigma^\tm g_1(u(t),v(t))dt\in W^\sigma:=(\omega-\delta_0)\cdot \Gamma^\sigma\ \text{ and }\ \left|\int_0^\sigma g_1(u(t),v(t))dt\right|\leq \delta\cdot d\]  for $u\in U_1$, $v\in U_2$. Thus, by the choice of $\delta$, we deduce that for all $u\in U_1$, $v\in U_2$ we have
$\alpha(u,v)=u_1+u_2$, where $u_1\in G_1\!\left(W^\sigma\right)$ and $|u_2|<\varepsilon$. Because $G_1\!\left(W^\sigma\right)$ is compact and $\varepsilon$ is arbitrarily small, we obtain $\alpha(U_1\times U_2)$ is relatively compact. Similarly we can proceed with $\beta$.
\end{proof}

\begin{rem}
The assumption that $0$ is not an atom of the measures $\mu_1,\mu_2$ follows from the observation from Example~\ref{ex:alpha=u(t_0)}, where it is shown that the Dirac measure $\delta_0$ is not admissible in our theory. Moreover, in the case $\mu_1(\{0\})>0$ or $\mu_2(\{0\})>0$, the initial conditions $u(0)=\alpha(u,v)$, $v(0)=\beta(u,v)$ become implicit, making the considerations unnecessarily complicated or incorrect.
Note also that the conclusion of Proposition~\ref{ex:alpha-integral} can be applied to the case of:
\begin{enumerate}[(i)]
	\item \emph{multi-point} conditions of the type
\[\alpha(u,v)=\sum_{s=1}^{k} \alpha_su(t_s),\quad \beta(u,v)=\sum_{s=1}^{r} \beta_sv(t'_s),\] where $0<t_1<\ldots<t_k$, $0<t'_1<\ldots<t'_r$ and $\alpha_s,\beta_s>0$.
	\item \emph{integral} conditions of the type
\[\alpha(u,v)=\int_0^\tm g_1(u(t))dt,\quad \beta(u,v)=\int_0^\tm g_2(v(t))dt.\]	
\end{enumerate}
\end{rem}

Let $\alpha,\beta$ be as in \eqref{init-con}.
Consider the assumption
\begin{equation}\label{eq:gG}
 p_1u\leq g_1(u,v)\leq q_1u,\ p_2v\leq g_2(u,v)\leq q_2v,\ P_iu\leq G_i(u)\leq Q_i u,
\end{equation}
for all $u,v\geq 0$, $i\in\{1,2\}$, where $0<p_i\leq q_i$, $0<P_i\leq Q_i$, $i\in\{1,2\}$. Similar assumptions were used, for nonlinear nonlocal conditions in the context of ODEs, in \cite{b:gi-caa,b:gipp-nonlin}.

Under the assumption \eqref{eq:gG} we have the following estimates:
\begin{equation}\label{eq:alfa,beta-estimates}
\begin{split}
p_1P_1 \int_0^{\tm} u(t)d\mu_1(t) \leq \alpha(u,v)\leq q_1Q_1\int_0^{\tm} u(t)d\mu_1(t),\\
p_2P_2 \int_0^{\tm} v(t)d\mu_2(t) \leq \beta(u,v) \leq q_2Q_2\int_0^{\tm} v(t)d\mu_2(t).
\end{split}\end{equation}

Put
\begin{equation}\label{eq:inf-fg}
f_{r,R}^0:=\inf\limits_{\substack{t\in[t_0,t_1],\ x\in D  \\ mr_1\leq u\leq R_1,\ 0\leq v\leq R_2}} \frac{f(t,x,u,v)}{r_1},\quad
g_{r,R}^0:=\inf\limits_{\substack{t\in[t_0,t_1],\ x\in D  \\ 0\leq u\leq R_1,\ mr_2\leq v\leq R_2}}   \frac{g(t,x,u,v)}{r_2},
\end{equation}

\begin{equation}\label{eq:inf'-fg}
f_{r,R}:=\inf\limits_{\substack{t\in[t_0,t_1],\ x\in D  \\ mr_1\leq u\leq R_1,\ mr_2\leq v\leq R_2}} \frac{f(t,x,u,v)}{r_1},\quad
g_{r,R}:=\inf\limits_{\substack{t\in[t_0,t_1],\ x\in D  \\ mr_1\leq u\leq R_1,\ mr_2\leq v\leq R_2}}   \frac{g(t,x,u,v)}{r_2}
\end{equation}

and
\begin{equation}\label{eq:sup-fg}
f^{R}:=\sup\limits_{\substack{t\in[0,\tm],\ x\in\Omega \\ 0\leq u\leq R_1,\ 0\leq v\leq R_2}}\frac{f(t,x,u,v)}{R_1},\quad
g^{R}:=\sup\limits_{\substack{t\in[0,\tm],\ x\in\Omega \\ 0\leq u\leq R_1,\ 0\leq v\leq R_2}}\frac{g(t,x,u,v)}{R_2}.
\end{equation}

Note that we have $f_{r,R}^0\leq f_{r,R}\leq  f^{R}$ and  $g_{r,R}^0\leq g_{r,R}\leq  g^{R}$.

\begin{lem}\label{lem:estimates} Let $\alpha,\beta$ be as in~\eqref{init-con}.

Assume that the inequalities \eqref{eq:gG} are satisfied.

\begin{enumerate}[\upshape (a)]
	\item If $|u|\leq R_1$, $|v|\leq R_2$, then we have \[
\frac{|N_1(u,v)|}{R_1}\leq q_1Q_1C_1^1+ f^R (q_1Q_1C_2^1 + C_1'),\quad
\frac{|N_2(u,v)|}{R_2}\leq q_2Q_2C_1^2+ g^R (q_2Q_2C_2^2 + C_1'),
\]
where the constants
\begin{align*}\label{eq:C1C2}
C_1^i=&\left|\int_0^\tm S_2(\tau)\chi_\Omega d\mu_i(\tau)\right|,\
C_1'=\left|\int_0^\tm S_2(\tau)\chi_\Omega d\tau\right|,\\
C_2^i=&\left|\int_0^\tm\int_0^t S_2(\tau)\chi_\Omega d\tau d\mu_i(t)\right|
\end{align*} are positive.
	\item The following implications hold:
\begin{align*}
\fl u\geq r_1\implies \frac{\fl{N_1(u,v)}}{r_1}\geq p_1P_1(c_1^1+f_{r,R}^0c_2^1),\\
\fl v\geq r_2\implies \frac{\fl{N_2(u,v)}}{r_2}\geq p_2P_2(c_1^2+g_{r,R}^0c_2^2),\end{align*}
where the constants
\begin{equation}\label{eq:c1c2}
c_1^i= \fl{\int_0^\tm S_2(\tau)\chi_D d\mu_i(\tau)},\
c_2^i= \fl{ \int_{t_0}^\tm\int_{t_0}^{\min\{t,t_1\}}S_2(t-\tau)\chi_D d\tau d\mu_i(t)}\end{equation} are positive.
\item  If $\fl u\geq r_1 $ and $\fl v\geq r_2$, then we have
\[\frac{\fl{N_1(u,v)}}{r_1}\geq p_1P_1(c_1^1+f_{r,R}c_2^1),\quad \frac{\fl{N_2(u,v)}}{r_2}\geq p_2P_2(c_1^2+g_{r,R}c_2^2),\]
\end{enumerate}{}
\end{lem}

\begin{proof}
Consider $u,v$ such that $|u|\leq R_1$, $|v|\leq R_2$.

(a) Using the symbols $u_0$, $\bar u$, $\bar v$ introduced in the proof of Proposition \ref{prop:fixed-sets} and having in mind the equality $u_0=\alpha(\bar u,\bar v)$ and the estimates \eqref{eq:alfa,beta-estimates} we obtain
\begin{equation}\label{eq:a1}
\frac 1{q_1Q_1}u_0  \leq \int_0^\tm\bar u(t)d\mu_1(t).
\end{equation}
Exploiting the equation \eqref{eq:t2} we obtain
\begin{equation}\label{eq:a2}
\int_0^\tm\bar u(t)d\mu_1(t) = \int_0^\tm S(t)u(0)d\mu_1(t)+\int_0^\tm\int_0^t S(t-\tau)F(u,v)(\tau)d\tau d\mu_1(t).
\end{equation}
From \eqref{eq:a1}, \eqref{eq:a2}, \eqref{eq:sup-fg} and the fact that $u\leq R_1\chi_\Omega$ it can be concluded that

\begin{equation}\label{eq:a3}\frac 1{q_1Q_1}u_0 \leq R_1\int_0^\tm S_2(t)\chi_\Omega d\mu_1(t)+R_1f^R\int_0^\tm\int_0^t S_2(\tau)\chi_\Omega d\tau d\mu_1(t).\end{equation}
From \eqref{eq:t1}, we obtain
\begin{equation}\label{eq:a4}
N_1(u,v)(t)=S(t)u_0+\int_0^tS(t-\tau)F(u,v)(\tau)d\tau\leq
S(t)u_0+R_1f^R\int_0^\tm S_2(\tau)\chi_\Omega d\tau.
\end{equation}
Combining \eqref{eq:a3} and \eqref{eq:a4} an applying the contractiveness of $S(t)$ gives
\[
\frac{|N_1(u,v)|}{R_1}\leq q_1Q_1C_1^1+ f^R (q_1Q_1C_2^1 + C_1').
\]

Similarly, we can obtain the estimate of $|N_2(u,v)|$.

(b) Assume now that $\fl u\geq r_1$. Then $u(0)\geq r_1\chi_D$. From \eqref{eq:Harnack} and \eqref{eq:inf-fg} we obtain
\[u(t)\geq mr_1\chi_D\text{ and } F(u,v)(t)\geq r_1f_{r,R}^0\chi_D \text{ for }t\in[t_0,t_1].\]

Using the symbols $\bar u,\bar v$ introduced in the proof of Proposition \ref{prop:fixed-sets} we obtain
\begin{align*}
\frac 1{p_1P_1}N_1(u,v)(0) & =\frac 1{p_1P_1} \alpha(\bar u,\bar v)\geq \int_0^\tm \bar u(t)d\mu_1(t) \\
& = \int_0^\tm S(t) u(0)d\mu_1(t)+ \int_0^\tm\int_0^t S(t-\tau)F(u,v)(\tau)d\tau d\mu_1(t)\geq \\
& \geq r_1\int_0^\tm S_2(t)\chi_Dd\mu_1(t)+ r_1f_{r,R}^0 \int_{t_0}^\tm\int_{t_0}^{\min\{t,t_1\}}S_2(t-\tau) \chi_D d\tau d\mu_1(t).
\end{align*}

Using the superadditivity of $\fl\cdot$ we obtain
\[\frac{\fl{N_1(u,v)}}{r_1}\geq p_1P_1(c_1^1+f_{r,R}^0c_2^1).\]
In the same manner we can obtain the estimate of $\fl{N_2(u,v)}$.

(c) Using the fact, that $\fl v\geq r_2$ implies $v(t)\geq mr_2\chi_D$ and following the calculations analogous to those above, we obtain the conclusion.

The positiveness of the constants $c_1^i,c_2^i,C_1^i,C_2^i,C_1'$, $i=1,2$, follows from Corollary \ref{cor:m>0}.
\end{proof}

\begin{rem}\label{rem:WhySuchHarnack}
As it was pointed out in Remark \ref{rem:Harnack}, the Harnack-type inequality \eqref{eq:Harnack} is an analogue of the inequality (3.4) of \cite{b:gi-mm-rp}. These two inequalities play a crucial role in obtaining the estimates from below and are utilized for the calculation of the fixed point index on some suitable subsets of a cone.

The difference between these two Harnack-type inequalities deserves a comment, as our choice to use the inequality \eqref{eq:Harnack} led us to build the new theory presented in Section~\ref{sect:abstr}. The inequality (3.4) of \cite{b:gi-mm-rp} was directly derived from a Harnack-type inequality given by Trudinger \cite{b:Trud-ell}. The natural counterpart in our context would be the \textit{parabolic} Harnack inequality by Trudinger \cite[Theorem 1.2]{b:Trud}, which is valid for all weak supersolutions $u$ of the equation $u_t-\Delta u=0$ (that is, functions $u\geq 0$ such that $u_t-\Delta u\geq 0$). This inequality could be expressed in the following manner:
\begin{equation}\label{eq:Harnack-bad}u\geq c\|u\|\chi_{Q_-},
\end{equation} where  $c>0$ is a constant and \[\|u\|=\int_{Q_+}|u(t,x)|dtdx,\ Q_+=\left[t^0,t^1\right]\times D,\ Q_-=\left[t_0,t_1\right]\times D\] for
$0<t^0<t^1<t_0<t_1\leq \tm$. However, the use of this inequality is somewhat unnatural and it seems that it leads to additional complication of the argument and to worse results. For the sake of brevity, we provide here only a brief explanation.

1) The inequality \eqref{eq:Harnack-bad} gives a lower bound on the values of $u$ on $Q_-$, which depend on the values of $u$ on $Q_+$. The proof of Lemma \ref{lem:estimates}(b) shows that it is more convenient to utilize the dependence of values $u|Q_-$ on values  of $u|\{0\}\times D$, due to the presence of the nonlocal boundary condition $u(0,x)=\alpha(u,v)(x)$.

2) The inequality \eqref{eq:Harnack} is actually a consequence of the inequality $u\geq m\fl u\chi_{[t_0,t_1]\times D}$ used for $u(t)=S_2(t)\chi_D$, which follows from the very definition of $m$. On the other hand, the inequality \eqref{eq:Harnack-bad} is more general in the sense that the constant $c$ is so chosen that a supersolution $u$ of the equation $u_t-\Delta u=0$ satisfies the estimate $u(t,x)\geq c$ on $Q_-$ \emph{whenever} $\|u\|\geq 1$. This universality, which in other context proves to be very important, is not exploited in our consideration, and unfortunately, it effects in a negative way the constants arising in the lower bounds of the nonlinearities (the counterparts of $c_1$ and $c_2$).  And lastly, those constants are more difficult to be established. In other words, having in mind the nature of the calculations from Lemma \ref{lem:estimates}, the use of minimum $\fl u$ is more convenient, natural and effective than the use of the integral seminorm $\|u\|$.
\end{rem}

\subsection{Existence results}
We are now prepared to establish some sufficient conditions for the existence of nonnegative nontrivial solutions of the problem \eqref{eq:parabolic}. In what follows we shall assume that $\alpha,\beta$ are as  in~\eqref{init-con}
and that the estimates \eqref{eq:gG} is satisfied.

\begin{thm}\label{thm:existence1 appl}
Assume there exist radii $0<r_i<R_i$, $i\in\{1,2\}$ such that
\begin{equation}\label{eq:ind1-appl}
q_1Q_1C_1^1+ f^R (q_1Q_1C_2^1 + C_1')\leq 1,\quad
q_2Q_2C_1^2+ g^R (q_2Q_2C_2^2 + C_1')\leq 1
\end{equation}
and
\begin{equation}\label{eq:ind0-appl}
p_1P_1(c_1^1+f_{r,R} c_2^1)>1,\quad
p_2P_2(c_1^2+g_{r,R} c_2^2)>1.
\end{equation}

Then there exists at least one nonnegative solution $(u,v)$ of the problem \eqref{eq:parabolic} such that $|u|\leq R_1,$
$|v|\leq R_2$ and  $\fl u  > r_1 $, $\fl v  > r_2$.

\end{thm}
\begin{proof} Lemma \ref{lem:estimates} implies that the conditions \eqref{eq:ind0} and \eqref{eq:ind1} are satisfied. Moreover,   $r_i<\fl{\psi_i} R_i$, since $\fl{\psi_i}=1$. The assertion follows from Theorem \ref{thm:existence 2a}.
\end{proof}

\begin{rem} It is worth mentioning that the condition $\fl u  > r_1$ implies $|u|> r_1$, which follows from both the definitions of $|\cdot|$ and $\fl\cdot$ and from Remark \ref{rem:lowerbound_norm}.
\end{rem}
\begin{thm}\label{thm:existence or appl}
 Let $0<r_1<R_1$, $0<r_2<R_2$ and let $\tilde R_1\leq R_1$, $\tilde R_2\leq R_2$. Define the quantities
\begin{equation*}\label{eq:inf''-fg}
f_{r,\tilde R}^{00}:=\inf\limits_{\substack{t\in[t_0,t_1],\ x\in D \\ 0\leq u\leq \tilde R_1,\ 0\leq v\leq \tilde R_2}} \frac{f(t,x,u,v)}{r_1},\quad
g_{r,\tilde R}^{00}:=\inf\limits_{\substack{t\in[t_0,t_1],\ x\in D \\ 0\leq u\leq \tilde R_1,\ 0\leq v\leq \tilde R_2}}   \frac{g(t,x,u,v)}{r_2}.
\end{equation*}

Assume that the condition \eqref{eq:ind1-appl} is satisfied and that
\begin{equation}
\label{eq:ind0 or appl}
f_{r,\tilde R}^{00}  \geq  \left(p_1P_1c_2^1\right)^{-1} \ \text{ or }\
g_{r,\tilde R}^{00}  \geq  \left(p_2P_2c_2^2\right)^{-1}.
\end{equation}
Then there exists a nontrivial nonnegative solution $(u,v)$ of \eqref{eq:parabolic} such that $ | u | \leq R_1$, $|v|\leq R_2$ and
\begin{equation}
\fl u  \geq r_1 \text{ or  }\fl v \geq r_2  \text{ or  } |u|>\tilde  R_1   \text{ or  }|v|>\tilde R_2.
\label{eq:conclusion ororor}
\end{equation}
In particular, if $\tilde R_1 =R_1 $ and $\tilde R _2 =R_2,$ then there exists a nontrivial nonnegative solution $(u,v)$ with either $ \fl u \geq r_1 $ or $ \fl v \geq r_2$.
\end{thm}

\begin{proof}
As in the previous proof we know that the condition \eqref{eq:ind1''} from Theorem~\ref{thm:existence or a} is satisfied. Let
\[(u,v)\in A:= \set{(u,v)}{\fl u<r_1,\ \fl v< r_2,\ |u|\leq \tilde R_1,\ |v|\leq \tilde R_2 }.\]

In the same way as in the proof of Lemma \ref{lem:estimates}(b) we can prove that
\begin{equation}\label{eq:estim(or)}
\frac{\fl{N_1(u,v)}}{r_1}\geq p_1P_1c_2^1 f_{r,\tilde R}^{00},\
\frac{\fl{N_2(u,v)}}{r_2}\geq p_2P_2c_2^2 g_{r,\tilde R}^{00}\ \text{ for } |u|\leq\tilde R_1\text{ and }|v|\leq\tilde R_2.
\end{equation}

Therefore, the assumption \eqref{eq:ind0 or appl} implies the condition \eqref{eq:ind0 or} is satisfied and we can apply Theorem \ref{thm:existence or a} to obtain a solution $(u,v)\not\in A$. Clearly, this is equivalent to \eqref{eq:conclusion ororor}.
\end{proof}

\begin{rem}
The importance of Theorem \ref{thm:existence or appl} consists in the fact that the
assumption \eqref{eq:ind0 or appl} involves only one component of the system
nonlinearity $(f,g).$ Therefore, it allows different kind of growth of $f$
and $g$ near the origin.
A similar remark also applies to the following theorem.
\end{rem}

\begin{thm} \label{thm:existence 3sols appl}
Assume there exist radii $0<\rho_i<r_i<R_i$, $i\in\{1,2\}$, such that
\[
q_1Q_1C_1^1+ f^R (q_1Q_1C_2^1 + C_1')\leq 1,\quad
q_2Q_2C_1^2+ g^R (q_2Q_2C_2^2 + C_1')\leq 1,
\]
\[
q_1Q_1C_1^1+ f^\rho (q_1Q_1C_2^1 + C_1')\leq 1,\quad
q_2Q_2C_1^2+ g^\rho (q_2Q_2C_2^2 + C_1')\leq 1
\]

and
\begin{equation}\label{eq:ind0-appl 3sols}
p_1P_1(c_1^1+f_{r,R} c_2^1)>1,\quad
p_2P_2(c_1^2+g_{r,R} c_2^2)>1
\end{equation}

Then there exist three nonnegative solutions $(u_i,v_i)$ ($i=1,2,3$) of the system \eqref{eq:parabolic} with
\begin{align*}
| u_1 | &<\rho _1,\ | v_1 | <\rho _2\  ( \text{possibly the zero solution}) ; \\
\fl{ u_2} &<r_1 \text{ or } \fl{ v_2} <r_2;\ | u_2| >\rho _1 \text{ or }|v_2| >\rho _2\ (\text{possibly one solution component zero}) ; \\
\fl{ u_{3}} &>r_1,\ \fl{ v_{3}} >r_2\  (\text{both solution components nonzero}) .
\end{align*}

By the following slight strengthening of the assumption~\eqref{eq:ind0-appl 3sols}:
\begin{equation}\label{eq:ind0-appl 3sols'}
p_1P_1(c_1^1+f_{r,R}^0 c_2^1)>1,\quad
p_2P_2(c_1^2+g_{r,R}^0 c_2^2)>1
\end{equation}
we obtain a slight improvement of the precision in localizing the second solution:
\[\fl{ u_2} <r_1,\ \fl{ v_2} <r_2;\ | u_2| >\rho _1 \ \text{or }|
v_2| >\rho _2.\]

Moreover, having given numbers $0<\varrho_i< \rho_i$ ($i=1,2$),

\emph{(i)} if
\[p_1P_1(c_1^1+f_{\varrho,\rho} c_2^1)>1,\quad
  p_2P_2(c_1^2+g_{\varrho,\rho} c_2^2)>1,\]
then $\fl{u_1} \geq \varrho_1$ and $\fl{v_1} \geq \varrho_2$;

\emph{(ii)} if
\[
f_{\varrho,\tilde \rho}^{00} \geq  \left(p_1P_1c_2^1\right)^{-1} \ \text{ or }\
g_{\varrho,\tilde \rho}^{00} \geq  \left(p_2P_2c_2^2\right)^{-1}
\]
for some $\tilde \rho_1\leq \rho _1,$ $\tilde{\rho}_2\leq \rho _2$, then $\fl{u_1} \geq \varrho_1$ or $\fl{ v_1} \geq \varrho_2$ or $|u_1|>\tilde \rho_1$ or $|v_1|>\tilde\rho_2$.
\end{thm}

\begin{proof} One can use Lemma \ref{lem:estimates}. The first assertion follows from Theorem \ref{thm:existence 3sols a}, while second follows from Theorem \ref{thm:existence 3sols a'}. The third part of the conclusion, i.e. assertions (i) and (ii), is a consequence of Theorem \ref{thm:existence 3sols a''} and \eqref{eq:estim(or)}.
\end{proof}

\subsection{Non-existence results} We now present some sufficient conditions for the non-existence of positive solutions of the system~\eqref{eq:parabolic}. We still assume that $\alpha,\beta$ are as in Example~\ref{ex:alpha-integral} and that \eqref{eq:gG} is satisfied. We also assume that $\tm\in\mathrm{supp}(\mu_i)$, $i=1,2$.

\begin{lem}\label{lem:fl>0} Let $(u,v)$ be a solution of the system~\eqref{eq:parabolic}. If $u\neq 0$ then $\fl{u}>0$ and if $v\neq 0$ then $\fl{v}>0$.
\end{lem}
\begin{proof} Let $0\leq t_0<\tm$ be such that $u(t_0)\neq 0$. From Proposition \ref{prop:fixed-sets} we know that $(u,v)$ is a fixed point of $M$. Thus,
\[\fl{u(0)}  =   \fl{\alpha(u,v)}\geq p_1P_1\fl{\int_0^\tm u(t)d\mu_1(t)}\geq p_1P_1\fl{\int_{t_0}^\tm u(t)d\mu_1(t)}\] and $u(t)\geq S(t-t_0)u(t_0)$ for $t\geq t_0$. From Lemma \ref{lem:m_eta>0} we therefore obtain
\[\fl{u}=\fl{u(0)}\geq p_1P_1 \fl{\int_{0}^{\tm-t_0} S(t)u(t_0) d\mu_1(t+t_0)}>0.\]
\end{proof}

\begin{thm}\label{thm:non-existence} Put
\[\overline e_i = \max\left(\frac{1-q_iQ_iC_1^i}{q_iQ_iC_2^i + C_1'},\frac{1-q_iQ_i\mu_i([0,\tm])}{C_1'} \right),\quad \underline e_i= \frac{(p_iP_i)^{-1}-c_1^i}{mc_2^i},\ i\in\{0,1\}.\]
 Assume that $(u_0,v_0)$ is a nonnegative solution of the system~\eqref{eq:parabolic}. If one of the following conditions holds:
\begin{align}
\label{eq:nonexistence u1} f(t,x,u,v)< \overline e_1 u &\text{ for all }t\in[0,\tm],\ x\in\Omega,\ u>0,\ v\geq 0,\\
\label{eq:nonexistence u2} f(t,x,u,v)> \underline e_1 u &\text{ for all }t\in[t_0,t_1],\ x\in D,\ u>0,\ v\geq 0,
\end{align}
then $u_0=0$.

Similarly, if one of the following conditions holds:
\begin{align}
\label{eq:nonexistence v1} g(t,x,u,v)< \overline e_2 v &\text{ for all }t\in[0,\tm],\ x\in\Omega,\ u\geq 0,\ v> 0,\\
\label{eq:nonexistence v2} g(t,x,u,v)> \underline e_2 v &\text{ for all }t\in[t_0,t_1],\ x\in D,\ u\geq 0,\ v> 0,
\end{align}
then $v_0=0$.

In particular, if $p_1P_1c_1^1>1$ ($p_2P_2c_1^2>1$), then $u_0=0$ ($v_0=0$), regardless of the properties of $f$  ($g$).
\end{thm}

\begin{proof} Suppose on the contrary that $u_0\neq 0$. Then $u_0=N_1(u_0,v_0)=M_1(u_0,v_0)$. Put $R_1=|u_0|$, $R_2=\max\{|v_0|,1\}$ and $r_1=\fl {u_0}>0$.

Assume that the inequality \eqref{eq:nonexistence u1}  holds.
Observe that $f(t,x,u,v)/R_1<\overline e_1$ for $t\in[0,\tm]$, $x\in\cl\Omega$, $0\leq u\leq R_1$ and $0\leq v\leq R_2$. Therefore, $f^R<\overline e_1$.
From Lemma \ref{lem:estimates}(a) we obtain
\begin{equation}\label{eq:estim 1}
1=\frac{|u_0|}{R_1}=\frac{|N_1(u_0,v_0)|}{R_1}\leq q_1Q_1C_1^1+ f^R (q_1Q_1C_2^1 + C_1').\end{equation}
Similarly, in the analogous manner as in the proof of Lemma \ref{lem:estimates}(a), one can show that
\begin{equation}\label{eq:estim 2}
1=\frac{|u_0|}{R_1}=\frac{|M_1(u_0,v_0)|}{R_1}\leq q_1Q_1 \mu_1([0,\tm])+ f^R C_1'.\end{equation}
The estimates \eqref{eq:estim 1} and \eqref{eq:estim 2} give $f^R\geq \overline e_1$, a contradiction.

Assume now that \eqref{eq:nonexistence u2}  holds.
Observe that $f(t,x,u,v)/r_1>m\underline e_1$  for $t\in[t_0,t_1]$, $x\in D$,  $mr_1\leq u\leq R_1$ and $0\leq v\leq R_2$. Therefore, $f_{r,R}^0>m\underline e_1$.

On the other hand, from Lemma \ref{lem:estimates}(b) we obtain
\begin{equation}\label{eq:estim 3} 1=\frac{\fl{u_0}}{r_1}=\frac{\fl{N_1(u_0,v_0)}}{r_1}\geq p_1P_1(c_1^1+f_{r,R}^0c_2^1).\end{equation}
From \eqref{eq:estim 3} we conclude that $f_{r,R}^0\leq m\underline e_1$, a contradiction.

The second assertion can be proved analogously.
\end{proof}

\begin{cor}\label{cor:non-existence}
\emph{(i)} If one of the inequalities \eqref{eq:nonexistence u1}-\eqref{eq:nonexistence v2}  holds, then there are no positive solutions of the system \eqref{eq:parabolic}.

\emph{(ii)} If one of the inequalities \eqref{eq:nonexistence u1}-\eqref{eq:nonexistence u2} holds and one of the inequalities \eqref{eq:nonexistence v1}-\eqref{eq:nonexistence v2} holds, then there are no
nontrivial nonnegative solutions of the system \eqref{eq:parabolic}.
\end{cor}

\subsection{Two examples}\label{ex:[0,pi]}
In the following one-dimensional examples we show that all the constants $C_1^i$, $C_1'$, $C_2^i$, $c_1^i$, $c_2^i$, $m$, $i=1,2$, that occur in our theory can be computed.

\begin{exa}\label{exa:ex1}
Let $\Omega=[0,\pi]$ and let $\mu_1=\mu_2$ be the Lebesgue measure on $[0,\tm]$. Let us put $D=[b,\pi-b]$ for a fixed $0< b< \pi/2$. Then
\[\chi_D=\sum_{k=0}^\infty \frac{4}{(2k+1)\pi}\cos\big((2k+1)b\big)\cdot \sin\big((2k+1)x\big) \] and
\[\chi_\Omega=\sum_{k=0}^\infty \frac{4}{(2k+1)\pi}\sin\big((2k+1)x\big),\] where the convergence is considered in the space $L^2(\Omega)$.

The evolution of $\chi_D$ and $\chi_\Omega$ is given by the formulae
\[(S_2(t)\chi_D)(x)=\sum_{k=0}^\infty \frac{4}{(2k+1)\pi}\cos\big((2k+1)b\big)\cdot e^{-(2k+1)^2t}\cdot  \sin\big((2k+1)x\big),\]
\[(S_2(t)\chi_\Omega)(x)=\sum_{k=0}^\infty \frac{4}{(2k+1)\pi}\cdot e^{-(2k+1)^2t}\cdot  \sin\big((2k+1)x\big).\]

Put $b=\pi/4$, $t_0=0$ and $t_1=\tm=1$. Then we have
\[m=(S_2(1)\chi_D)\left(\frac\pi 4\right)=
\sum_{k=0}^\infty (-1)^k\frac2{(2k+1)\pi} e^{-(2k+1)^2}\approx 0.23.
\]

Since we consider the Lebesgue measure in the definition of the operators $\alpha$ and $\beta$, we shall omit the superscripts in the constants $C_j^i$, $c_j^i$.

Because
\[
\int_0^\tm S_2(\tau)\chi_Dd\tau(x)=\sum_{k=0}^\infty \frac{4}{(2k+1)^3\pi}\cos(2k+1)\frac\pi 4\cdot \left(1-e^{-(2k+1)^2}\right)\cdot  \sin\big((2k+1)x\big)
\]
and
\[
\int_{0}^\tm\int_{0}^{t}S_2(t-\tau) \chi_D d\tau dt=
\sum_{k=0}^\infty \frac{4}{(2k+1)^5\pi}\cos(2k+1)\frac\pi 4\cdot   e^{-(2k+1)^2}\cdot  \sin\big((2k+1)x\big),
\]
we obtain
\[
c_1=\sum_{k=0}^\infty \frac{2\cdot (-1)^k}{(2k+1)^3\pi}\left(1-e^{-(2k+1)^2}\right)\approx 0.38\]
and
\[
c_2=\sum_{k=0}^\infty \frac{2\cdot (-1)^k}{(2k+1)^5\pi}   e^{-(2k+1)^2} \approx 0.23
\]

Moreover, because

\[\int_0^\tm S_2(\tau)\chi_\Omega d\tau =\sum_{k=0}^\infty \frac{4}{(2k+1)^3\pi}\cdot \left(1-e^{-(2k+1)^2}\right)\cdot  \sin\big((2k+1)x\big)\] and

\[
\int_0^\tm\int_0^t S_2(\tau)\chi_\Omega d\tau dt=
\sum_{k=0}^\infty \frac{4}{(2k+1)^5\pi}\cdot   e^{-(2k+1)^2}\cdot  \sin\big((2k+1)x\big),
\]

\begin{equation}C_1=\left|\int_0^\tm S_2(\tau)\chi_\Omega d\tau\right|,\quad C_2=\left|\int_0^\tm\int_0^t S_2(\tau)\chi_\Omega d\tau dt\right|\end{equation}

we obtain
\[
C_1=C_1'= \sum_{k=0}^\infty \frac{4\cdot (-1)^k}{(2k+1)^3\pi}\cdot \left(1-e^{-(2k+1)^2}\right)\approx 0.77
\]

and

\[
C_2=
\sum_{k=0}^\infty \frac{4\cdot (-1)^k}{(2k+1)^5\pi}\cdot   e^{-(2k+1)^2}\approx 0.47.
\]

If we put $G_i(x)=g_i(x)=x$, then $p_i=P_i=q_i=Q_i=1$ and the conditions \eqref{eq:ind1-appl} and \eqref{eq:ind0-appl} are equivalent to the following inequalities:

\[
  f^R,g^R\leq \frac{1-C_1}{C_1' + C_2}\approx 0.19,\quad
	f_{r,R},g_{r,R} >\frac{1-c_1}{c_2}  \approx 2.64.
\]

The numerical calculations indicate that the choice $b=\pi/4$ is optimal, i.e. the ratio $(1-c_1)/(c_2m)$ is the smallest.
\end{exa}
\begin{exa}\label{exa:ex2}
Let $\Omega,D,\tm$  be as in Example \ref{exa:ex1}. We consider the problem \eqref{eq:parabolic-intro} with conditions of the form $u(0)=\kappa u(\tm)$, $v(0)=\kappa v(\tm)$, $\kappa>0$. We discuss in details the cases of periodic conditions and of solutions that double in time.
 Therefore $u(0)=\alpha(u,v)$ and $v(0)=\beta(u,v)$, where $\alpha$ and $\beta$ are as 
in \eqref{init-con} with $G_i=\mathrm{id}$, $g_i(u_1,u_2)=\kappa\cdot u_i$ and $\mu_1=\mu_2=\delta_{\tm}$ is a Dirac measure. Hence we have $p_i=q_i=\kappa$ and $P_i=Q_i=1$. Put $t_0=0$ and $t_1=\tm=1$.

The evolution of $\chi_D$ and $\chi_\Omega$ is the same as in Example \ref{exa:ex1}. Therefore, $m\approx 0.23$. The remaining constants can be calculated by using the formulae
\[
C_1=\left|S_2(\tm)\chi_\Omega\right|,\ C_2=C_1'=\left|\int_0^\tm S_2(\tau)\chi_\Omega d\tau\right|,\]
\[
c_1=\fl{S_2(\tm)\chi_D},\ c_2=\fl{\int_0^\tm S_2(\tau)\chi_D d\tau}
\]
and the calculations from Example \ref{exa:ex1}. We obtain

\[C_1 = \sum_{k=0}^\infty(-1)^k\frac{4}{(2k+1)\pi}e^{-(2k+1)^2}=2m\approx 0.46, \ C_2\approx 0.77,\ c_1=m\approx 0.23,\ c_2\approx 0.38.\]

Finally, the conditions \eqref{eq:ind1-appl} and \eqref{eq:ind0-appl} are equivalent to the following inequalities:

\[
  f^R,g^R\leq \mathbf{C}(\kappa):=\frac{1-\kappa C_1}{C_1' + \kappa C_2},\quad
	f_{r,R},g_{r,R} > \mathbf{c}(\kappa):=\frac{1-\kappa c_1}{\kappa  c_2}.
\]
For instance, in the periodic case ($\kappa=1$) we obtain $\mathbf{C}(1)\approx 0.35$ and  $\mathbf{c}(1)\approx 2.03$. Moreover, in the case of solutions that double in time ($\kappa=0.5$) we obtain $\mathbf{C}(0.5)\approx 0.67$ and $\mathbf{c}(0.5)\approx 4.66$.

\end{exa}

\section*{Acknowledgements}
The authors wish to thank the anonymous Referee for the constructive comments.
The authors are indebted to Prof. Precup for the useful discussions on the nonlocal parabolic case, in particular for providing them with the works \cite{b:tc-rp-pr,b:rp-par}.
G. Infante was partially supported by G.N.A.M.P.A. - INdAM (Italy).
This paper was written during the postdoctoral
stage of M. Maciejewski at the University of Calabria, supported by a
research fellowship within the project ``Enhancing Educational Potential of
Nicolaus Copernicus University in the Disciplines of Mathematical and
Natural Sciences'' (Project no. POKL.04.01.01-00-081/10) and
 by the NCN Grant 2013/09/B/ST1/01963.

\end{document}